\documentclass[graybox]{svmult}

\usepackage{amsfonts,amsmath,amstext,latexsym,amssymb}
\usepackage{mathptmx}       
\usepackage{helvet}         
\usepackage{courier}        
\usepackage{type1cm}        
\usepackage{cite}                            
\usepackage{makeidx}         
\usepackage{graphicx}        
\usepackage{multicol}        
\usepackage[bottom]{footmisc}
\makeindex             
\usepackage{enumitem}  
\usepackage{calc}
\setlist{labelindent=1pt,itemsep=0.1cm}
\setlist[itemize]{leftmargin=0.5cm}
\setlist[enumerate]{itemindent=0em,leftmargin=0.5cm}
\setlist[enumerate,1]{label={\upshape\arabic*)}}

\usepackage{float} 

\smartqed

\allowdisplaybreaks

\newtheorem{thm}{Theorem}[section]
\newtheorem{cor}[thm]{Corollary}
\newtheorem{ex}[thm]{Example}

\newcommand{\di}{\displaystyle}

\DeclareMathOperator{\esssup}{ess\;sup}

\begin{document}

\title*{Differential and linear integral operators representing polynomial covariance commutation relations in $C^\infty$}
\titlerunning{Differential and integral operators representing covariance commutation relations}
\author{Domingos Djinja, Sergei Silvestrov, Alex Behakanira Tumwesigye}
\institute{Domingos Djinja \at Department of Mathematics and Informatics, Faculty of Sciences, Eduardo Mondlane University, Box 257, Maputo, Mozambique. \email{domingos.djindja@uem.ac.mz}
\at Division of Mathematics and Physics, UKK, M\"{a}lardalens University, Box 883, V\"{a}ster\.{a}s, Sweden.
\and Sergei Silvestrov \at Division of Mathematics and Physics, The School of Education, Culture and Communication, M\"{a}lardalens University, Box 883, V\"{a}ster\.{a}s, Sweden. \email{sergei.silvestrov@mdu.se}
\and Alex Behakanira Tumwesigye \at Department of Mathematics, College of Natural Sciences, Makerere University, Box 7062, Kampala, Uganda. \email{alex.tumwesgye@mak.ac.ug}}
%
%

	
\maketitle\label{chapDjinjaSilvestrovTumwesigyeVolteraops}

\abstract*{Representations of polynomial covariance commutation relations by pairs of linear integral and differential operators
are constructed in the space of infinitely continuously differentiable functions.
Representations of polynomial covariance commutation relations by pairs consisting of a differential and linear integral operator are considered including conditions on kernels and coefficients of operators, and examples.
\keywords{integral operator, covariance commutation relation} \\
{\bf MSC 2020:} 47G10, 47L80, 47L65, 47L10}

\abstract{Representations of polynomial covariance commutation relations by pairs of linear integral and differential operators
are constructed in the space of infinitely continuously differentiable functions.
Representations of polynomial covariance commutation relations by pairs consisting of a differential and linear integral operator are considered including conditions on kernels and coefficients of operators, and examples.
\keywords{integral operator, covariance commutation relation, differential operator} \\
{\bf MSC 2020:} 47G10, 47L80, 47L65, 47L10}

\section{Introduction}

In many areas of applications there may be found relations of the form
\begin{gather} \label{covCommutationRelation}
  AB=B F(A)
\end{gather}
where $A, B$ are elements of
an associative algebra and  $F:\mathbb{R}\to\mathbb{R}$ is a function satisfying certain conditions.
Such commutation relations are usually called covariance relations, crossed product relations or semi-direct product relations. Elements of an algebra that satisfy \eqref{covCommutationRelation} are called a representation of this relation.  Representations of covariance commutation relations \eqref{covCommutationRelation} by linear operators are important for the study of actions and induced representations of groups and semigroups, crossed product operator algebras, dynamical systems, harmonic analysis, wavelets and fractals analysis and, hence have applications in physics and engineering
\cite{BratJorgIFSAMSmemo99,BratJorgbook,JorgWavSignFracbook,JorgOpRepTh88,JorMoore84,MACbook1,MACbook2,MACbook3,OstSambook,Pedbook79,Samoilenkobook}.
A pair $(A,B)$ of elements in the corresponding
associative algebra that satisfies \eqref{covCommutationRelation} is called a representation of this relation
\cite{Samoilenkobook}.
Algebraic properties of the commutation relation \eqref{covCommutationRelation} are  important in description of properties of its representations. For instance, there is a well known link between linear operators satisfying the commutation  relation \eqref{covCommutationRelation} and spectral theory \cite{Samoilenkobook}.

A description of the structure of representations for the relation \eqref{covCommutationRelation} and more general families of self-adjoint operators satisfying such relations by bounded and unbounded self-adjoint linear operators on a Hilbert space using reordering formulas for functions of the algebra elements and operators satisfying covariance commutation relation, functional calculus and spectral representation of operators and interplay with dynamical systems generated by iteration of maps involved in the commutation relations have been considered in  \cite{BratEvansJorg2000,CarlsenSilvExpoMath07,CarlsenSilvAAM09,CarlsenSilvProcEAS10,DutkayJorg3,DJS12JFASilv,DLS09,DutSilvProcAMS,DutSilvSV,JSvT12a,JSvT12b,Mansour16,JMusondaPhdth18,JMusonda19,Musonda20,Nazaikinskii96,OstSambook,PerssonSilvestrov031,PerssonSilvestrov032,PersSilv:CommutRelDinSyst,RST16,RSST16,Samoilenkobook,SaV8894,SilPhD95,STomdynsystype1,SilWallin96,SvSJ07a,SvSJ07b,SvSJ07c,SvT09,Tomiyama87,Tomiama:SeoulLN1992,Tomiama:SeoulLN2part2000,AlexThesis2018,VaislebSa90}.


Constructions of representations of polynomial covariance commutations relations by pairs of  linear integral operators with general kernels, linear integral operators with separable kernels and linear multiplication operator  defined on some Banach spaces have been considered in \cite{DjinjaEtAll_IntOpOverMeasureSpaces,DjinjaEtAll_LinItOpInGenSepKern,DjinjaEtAll_LinMultIntOp}.

In this paper, we construct representations of \eqref{covCommutationRelation} by pairs of  linear integral and differential operators in the space of infinitely continuously differentiable functions. Such representations
can also be viewed as solutions for integro differential operator equations $AX=XF(A)$, when $A$ is specified or
$XB=BF(X)$ when $B$ is specified.    \par

This paper is organized in three sections. After the introduction, we present  preliminaries in Section \ref{SecPreNot}.
  In Section \ref{SecIntOpDiffOp} we present results of representations of commutation relation \eqref{covCommutationRelation} by pairs of differential and  linear integral operator.

%

\section{Preliminaries and notations}\label{SecPreNot}
Let $\mathbb{R}$ be the set of all real numbers,  $ X$ be a non-empty space, and let $S\subseteq X$, $(S,\Sigma, \mu)$ be a $\sigma$-finite measure space, where $\Sigma$ is a $\sigma-$algebra with measurable subsets of $S$,
and $S$ can be covered with at most countable many disjoint sets $E_1,E_2,E_3,\ldots$ such that $ E_i\in \Sigma, \,
\mu(E_i)<\infty$, $i=1,2,\ldots$  and $\mu$ is a measure.
For $1\leq p<\infty,$ the linear space $L_p(S,\mu)$ of all classes of equivalent (different at most on a set of zero measure)
measurable functions $f:S\to \mathbb{R}$ such that
$\int\limits_{S} |f(t)|^p d\mu < \infty,$ is a Banach space (Hilbert space when $p=2$) with norm
$\| f\|_p= \left( \int\limits_{S} |f(t)|^p dt \right)^{\frac{1}{p}}.$
The linear space $L_\infty(S,\mu)$ of all classes of equivalent (different at most on a set of zero measure) measurable functions $f:S\to \mathbb{R}$ such that, $|f(t)|\leq \lambda$ for some $\lambda=\lambda_f>0$, almost everywhere. This is a Banach space with norm
$\|f\|_{\infty}=\mathop{\esssup}\limits_{t\in S} |f(t)|.$
The support of a function $f:\, X\to\mathbb{R}$ is the set
${\rm supp \, } f = \{t\in X:\, f(t)\not=0\}$.
We will also use the following notation
$ Q_{G}(u,v)=\int\limits_{G} u(t)v(t)d\mu$
where $G\in \Sigma$ and $u,v$ are such functions $u,v:\, G\to \mathbb{R} $ that integral on the right hand side exist and is finite.
(sometimes, $G$ or $v,u$ will be omitted for convenience of exposition).
For $\alpha, \beta \in \mathbb{R}$, the linear space of all continuous functions $f:[\alpha,\beta]\to \mathbb{R}$ is a Banach
space with norm $\|f\|=\max\limits_{t\in [\alpha,\beta]} |f(t)|$.
The linear space of all infinitely continuously differentiable functions  $f:[\alpha,\beta]\to \mathbb{R}$ is denoted by $C^{\infty}[\alpha,\beta]$. For more details on these and other related basic definitions and notations we refer to \cite{AdamsG,AdamsSobolevSpaces,BrezisFASobolevSpaces, FollandRA,Kantarovitch,Kolmogorov,KolmogorovVol2,RudinRCA, RynneLFA}.


We present a useful lemma for continuous linear functionals.
\begin{lemma}\label{LemmaCLFVanishCompleteSetVanishUniverse}
Let $E$ be a normed vector space. Let $H$ be a continuous linear functional defined on $E$ and
$M\subseteq E$ be such that  cl(span($M$))=$E$, where cl$(\Pi)$ is the closure of the set $\Pi$. If $H(x)=0$ for all $x\in M$ then
$H(x)=0$ for all $x\in E$.
\end{lemma}

\begin{proof}
Let $x\in $ span$(M)$, such that $x=\sum\limits_{i=1}^{m} \alpha_i x_i$, where $x_i\in M$, $\alpha_i\in\mathbb{R}$, for all $1\leq i\leq m$. Then, by using linearity of $H$ we have
\begin{equation*}
  H(x)=H\left(\sum\limits_{i=1}^{m} \alpha_i x_i\right)=\sum\limits_{i=1}^{m}\alpha_i H(x_i)=0.
\end{equation*}
Then $H(x)=0$ for all $x\in$ span($M$). Suppose that $y\in E$, and since cl(span$(M)$)$=E$, there exist
elements $y_n\in$ span($M$) such that $y_n\to y$, when $n\to +\infty$, that is, $\|y_n-y\|_E\to 0,$
$n\to +\infty$. Therefore, by using linearity and continuity of $H$ we have
\begin{equation*}
  H(y)=H\left(\lim_{n\to +\infty} y_n\right)=\lim_{n\to+\infty} H(y_n)=0.
\end{equation*}
Hence, $H(y)=0$ for all $y\in E$, that is, $H$ is the zero functional.
\qed\end{proof}

\section{Representations by pairs of integral and differential operator}\label{SecIntOpDiffOp}
Let
\begin{gather}\label{SetDCinfinityXofalphabetazero}
D_{[\alpha,\beta]}=\left\{ x\in C^{\infty}[\alpha,\beta]:\ x(\alpha)=x(\beta)=0  \right\},
\end{gather}
where $C^{\infty}[\alpha,\beta]$ is the space of infinitely continuously differentiable functions in $[\alpha,\beta]$.

\begin{theorem}\label{thmIntDiffOp}
Consider linear operators $A:C^\infty[\alpha,\beta]\to C^\infty[\alpha,\beta]$,  $B:C^\infty[\alpha,\beta]\to C^\infty[\alpha,\beta]$
defined as follows
\begin{gather*}
  (Ax)(t)= \int\limits_{\alpha}^{\beta} k(t,s)x(s)ds,\quad (Bx)(t)= b(t)\frac{dx}{dt},
\end{gather*}
where $\alpha,\beta$ are real numbers, $k(\cdot, \cdot)\in C^{\infty}([\alpha,\beta]^2)$, $ b(\cdot)\in C^{\infty}[\alpha,\beta]$.
Consider a polynomial $F(t)=\sum\limits_{i=0}^{n}\delta_i t^i$, $\delta_0,\ldots,\delta_n\in \mathbb{R}$. Let
\begin{eqnarray*}
  k_0(t,s)&=&k(t,s),\quad k_1(t,s)=\int\limits_{\alpha}^{\beta} k(t,\tau)k(\tau,s)d\tau,\\
  k_m(t,s)&=&\int\limits_{\alpha}^{\beta} k(t,\tau)k_{m-1}(\tau,s)d\tau,\quad m=1,\ldots,n, \\
  F_m(k(t,s))&=&\sum_{j=1}^{m} \delta_j k_{j-1}(t,s),\quad m=1,\ldots,n.
\end{eqnarray*}
Then  for all $x\in C^{\infty}[\alpha,\beta]$ we have  $ABx=BF(A)x $ if and only if
\begin{eqnarray}\nonumber
& k(t,\beta)b(\beta)x(\beta)-k(t,\alpha)b(\alpha)x(\alpha) -\int\limits_{\alpha}^{\beta}\frac{\partial [b(s)k(t,s)]}{\partial s}x(s)ds \\ \label{CondAdiffBInfOPInWholeSpace}
& =
\delta_0 b(t)x'(t)+ b(t)\int\limits_{\alpha}^{\beta}\frac{\partial F_n(k(t,s))}{\partial t}x(s)ds.
\end{eqnarray}

If $x\in D_{[\alpha,\beta]}$,  where $D_{[\alpha,\beta]}$ is given in \eqref{SetDCinfinityXofalphabetazero}, then condition \eqref{CondAdiffBInfOPInWholeSpace} is equivalent to
\begin{gather}\label{CondAdiffBIntOp}
    -\int\limits_{\alpha}^{\beta}\frac{\partial [b(s)k(t,s)]}{\partial s}x(s)ds=\delta_0 b(t)x'(t) + b(t)\int\limits_{\alpha}^{\beta}\frac{\partial F_n(k(t,s))}{\partial t}x(s)ds.
\end{gather}
In particular, if $\delta_0=0$, that is, $F(t)=\sum\limits_{i=1}^{n}\delta_i t^i$ then condition \eqref{CondAdiffBIntOp} can be simplified, that is, for all $x\in D_{[\alpha,\beta]}$ the equality $(ABx)(t)=(BF(A)x)(t)$ holds if and only if
      \begin{gather}\label{SystemEqIntDiffOp1}
       -\frac{\partial [b(s)k(t,s)]}{\partial s}=b(t)\frac{\partial F_n(k(t,s))}{\partial t},\quad (t,s)\in [\alpha,\beta]^2.
        \end{gather}
\end{theorem}

\begin{proof}
Operators $A$, $B$ are well defined and so are the compositions $AB$, $BA$, $A^n$, $BA^n$ for any nonnegative integer $n$.
By applying Fubini Theorem on changing the order of integration in the composition and proceeding by induction, we get
  \begin{gather*}
  (A^nx)(t)=\int\limits_{\alpha}^{\beta} k_{n-1}(t,s)x(s)ds,\quad n\in\mathbb{N},
  \end{gather*}
  where
  \begin{eqnarray*}
  k_1(t,s)&=&\int\limits_{\alpha}^{\beta} k(t,\tau)k(\tau,s)d\tau, \quad k_0(t,s)=k(t,s),\\
  k_n(t,s)&=&\int\limits_{\alpha}^{\beta} k(t,\tau)k_{n-1}(\tau,s)d\tau,\quad n\in\mathbb{N} .
  \end{eqnarray*}
Thus we have
  \begin{eqnarray*}
   (F(A)x)(t)&=&\delta_0 x(t)+ \sum\limits_{j=1}^{n} \delta_j (A^j x)(t)=\delta_0 x(t)+\sum\limits_{j=1}^{n} \delta_{j} \int\limits_{\alpha}^{\beta} k_{j-1}(t,s)x(s)ds \\
   &=&\delta_0 x(t)+\int\limits_{\alpha}^{\beta} F_n(k(t,s))x(s)ds,
   \end{eqnarray*}
 where
 \begin{gather*}
   F_n(k(t,s))=\sum_{j=1}^{n} \delta_j k_{j-1}(t,s),\quad n\in\mathbb{N}
  \end{gather*}
  and
  \begin{eqnarray*}
    (BF(A)x)(t)&=&\delta_0 b(t)x'(t)+ b(t)\int\limits_{\alpha}^{\beta}\frac{\partial F_n(k(t,s))}{\partial t}x(s)ds\\
    (ABx)(t)&=&\int\limits_{\alpha}^{\beta} k(t,s)(Bx)(s)ds=\int\limits_{\alpha}^{\beta} k(t,s)b(s)x'(s)ds
      \\
 &=& k(t,\beta)b(\beta)x(\beta)-k(t,\alpha)b(\alpha)x(\alpha)-\int\limits_{\alpha}^{\beta}\frac{\partial [b(s)k(t,s)]}{\partial s}x(s)ds.
  \end{eqnarray*}
 Then, for $x\in C^{\infty}[\alpha,\beta]$ we have $(ABx)(t)=(BF(A)x)(t)$ if and only if condition \eqref{CondAdiffBInfOPInWholeSpace} holds.
 Since  for $ x(\cdot)\in D_{[\alpha,\beta]}$ we have $x(\beta)=x(\alpha)=0$, then for all $x\in D_{[\alpha,\beta]}$ the equality $(ABx)(t)=(BF(A)x)(t)$ holds true if and only if condition \eqref{CondAdiffBIntOp} also holds.
If $\delta_0=0$ then condition \eqref{CondAdiffBIntOp} reduces to the following:
 \begin{gather}\label{EqIntBothSidesRepreADiffBIntOp}
 \forall\, x\in D_{[\alpha,\beta]}:\quad  -\int\limits_{\alpha}^{\beta}\frac{\partial [b(s)k(t,s)]}{\partial s}x(s)ds=b(t)\int\limits_{\alpha}^{\beta}\frac{\partial F_n(k(t,s))}{\partial t}x(s)ds.
 \end{gather}
It follows from \cite[Corollary 4.23]{BrezisFASobolevSpaces} and Lemma \ref{LemmaCLFVanishCompleteSetVanishUniverse} that condition \eqref{EqIntBothSidesRepreADiffBIntOp} is equivalent to
\begin{gather*}
  -\frac{\partial [b(s)k(t,s)]}{\partial s}=b(t)\frac{\partial F_n(k(t,s))}{\partial t},\quad (t,s)\in\, [\alpha,\beta]^2. \tag*{\qed}
\end{gather*}
\end{proof}

\begin{lemma}\label{LemmaContinuousFunctionsSupportMeasureZeroIntOpDiffOP}
Let $a,b\in C[\alpha,\beta]$, $\alpha,\beta\in\mathbb{R}$. We put
\begin{gather*}
  \Omega_{ab}={\rm supp }\,a\, \cap \, {\rm supp }\, b,
\end{gather*}
where ${\rm supp}\, a$ is the support of the function $a$.
If the set $\Omega_{ab}$ has Lebesgue measure zero then it is empty.
\end{lemma}

\begin{proof}
  Without loss of generality we suppose that there exists $\alpha_0\in (\alpha,\beta)$ such that
  $\alpha_0\in \Omega_{ab}$, that is, $\Omega_{ab}$ is not empty. Then, $a(\alpha_0)\not=0$ and $b(\alpha_0)\not=0$.
  Since $a,b$ are continuous functions, there is an open interval $V_{\alpha_0}=(\alpha_0-\varepsilon,\alpha_0+\varepsilon)\subset (\alpha,\beta)$, for some $\varepsilon>0$ such that for all $t\in V_{\alpha_0}$,
     $
       a(t)\not=0, b(t)\not=0.
     $
    Then  the Lebesgue measure of $\Omega_{ab}$ is positive. This contradicts the hypothesis. Then, $\Omega_{ab}=\emptyset$. \qed
\end{proof}

\begin{cor}\label{corCorrespondingthmIntDiffOp1.1}
Consider linear operators $A:C^\infty[\alpha,\beta]\to C^\infty[\alpha,\beta]$,  $B:C^\infty[\alpha,\beta]\to C^\infty[\alpha,\beta]$  defined as follows
\begin{equation}\label{IntegralOperatorDegeneratedKernelDiffOp}
  (Ax)(t)= \int\limits_{\alpha}^{\beta} a(t)c(s)x(s)ds,\quad (Bx)(t)= b(t)\frac{dx}{dt},
\end{equation}
where $\alpha,\beta\in \mathbb{R}$ and $a, b, c:[\alpha,\beta]\to\mathbb{R}$ belong to $C^\infty[\alpha,\beta]$.
Let $F(t)=\sum\limits_{i=1}^{n}\delta_i t^i$, where
$\delta_1,\ldots,\delta_n \in \mathbb{R}$. Set
\begin{gather}\label{SetOmegaaOmegacProofCorAdiffBIntOp}
  \Omega_a=\{t\in[\alpha,\beta]:\, a(t)\not=0,\ a'(t)\not=0\},\quad \Omega_c=\{s\in[\alpha,\beta]:\, c(s)\not=0\}, \\
  \label{constK1ProofCorAdiffBIntOp}
  k_1=\sum\limits_{k=1}^{n} \delta_k \big(\int\limits_{\alpha}^{\beta} a(s)c(s)ds\big)^{k-1}.
\end{gather}
 Suppose that the function $a$ is given and  $\Omega_a=\Omega_c$ with exception perhaps of countably many points in $[\alpha,\beta]$. Then,
  for all $x\in D_{[\alpha,\beta]}$, where $D_{[\alpha,\beta]}$ is given in \eqref{SetDCinfinityXofalphabetazero}, that is $D_{[\alpha,\beta]}=\{x\in C^\infty[\alpha,\beta]:\, x(\alpha)=x(\beta)=0\}$, the equality $ABx=BF(A)x$ is satisfied
if and only if one of the following is true:
\begin{enumerate}[label=\textup{\arabic*.}, ref=\arabic*]
        \item if $k_1=0$ and there exists $t\in [\alpha,\beta]$ such that $a(t)\not=0$ then $b(s)c(s)=\gamma_0$ for some real constant $\gamma_0$.
       \item if $k_1\not=0$ then the following conditions are fulfilled
     \begin{enumerate}[label=\textup{\alph*)}]
         \item
     if $(t,s)\in \Omega_a\times \Omega_c$ then
   \begin{gather*}
     b(t)=\frac{a(t)\lambda}{k_1a'(t)},\\
     c(s)=c_{\xi_0}(s)={\rm exp}\big(\di\int\limits_{[\xi_0,s]\cap\,\Omega_a\,\cap\, \Omega_c} \frac{-k_1a'(\tau)^2+a(\tau)a''(\tau)-a'(\tau)^2}{a(\tau)a'(\tau)}d\tau\big),
   \end{gather*}
    for some  $\lambda\in\mathbb{R}\setminus\{0\}$ and some $\xi_0\in \Omega_a\cap \Omega_c$.

    \item if $(t,s)\in \Omega_a\times ([\alpha,\beta]\setminus \Omega_c)$ then  the set
    \begin{gather*}
       {\rm supp }\, c'\ \cap \, {\rm supp}\, b\, \cap\, [\alpha,\beta]\setminus \Omega_c  =\emptyset .
    \end{gather*}

    \item if $(t,s)\in ([\alpha,\beta]\setminus\Omega_a)\times \Omega_c$ then either the set
    \begin{gather*}
      {\rm supp }\, a' \, \cap \, {\rm supp}\, b\, \cap\,[\alpha,\beta]\setminus\Omega_a=\emptyset
    \end{gather*}
     or ${\rm supp }\,a'\, \setminus\, {\rm supp}\, a\not=\emptyset$ implies $b(s)c(s)=\gamma_0$ for some constant $\gamma_0\in\mathbb{R}$.

    \item if $(t,s)\in ([\alpha,\beta]\setminus\Omega_a)\times ([\alpha,\beta]\setminus\Omega_c)$ then
          ${\rm supp }\,a'\, \setminus\, {\rm supp}\, a\not=\emptyset$ implies that
          \begin{gather*}
             {\rm supp }\, c' \, \cap \, {\rm supp}\, b\, \cap\,[\alpha,\beta]\setminus\Omega_c=\emptyset.
          \end{gather*}
   \end{enumerate}
  \end{enumerate}
\end{cor}

\begin{proof}
Let $k_0(t,s)=a(t)c(s)$. Further,
\begin{gather*}
  k_1(t,s)=\int\limits_{\alpha}^{\beta} k_0(t,\tau)k_0(\tau,s)d\tau=\int\limits_{\alpha}^{\beta} a(t)c(\tau)a(\tau)c(s)d\tau=a(t)c(s)\int\limits_{\alpha}^{\beta}c(\tau)a(\tau)d\tau=\\
  =k_0(t,s)\int\limits_{\alpha}^{\beta} k_0(\tau,\tau)d\tau,\\
  k_2(t,s)=\int\limits_{\alpha}^{\beta} k_0(t,\tau)k_1(\tau,s)d\tau= \int\limits_{\alpha}^{\beta} k_0(t,\tau)k_1(\tau,s)d\tau
  =k_0(t,s)\big(\int\limits_{\alpha}^{\beta} k_0(\tau,\tau)d\tau\big)^2.
\end{gather*}
Therefore, for $m= 1,\ldots,n$ we have
  \begin{gather*}
  k_m(t,s)=k_0(t,s)\big(\int\limits_{\alpha}^{\beta} k_0(\tau,\tau)d\tau\big)^m=a(t)c(s)\big(\int\limits_{\alpha}^{\beta} a(\tau)c(\tau)d\tau\big)^m.
  \end{gather*}
  It follows that,  for $m= 1,\ldots,n$ we have
  \begin{gather*}
  F_m(k(t,s))=\sum_{j=1}^{m} \delta_j k_{j-1}(t,s)=\sum_{j=1}^{m} \delta_j a(t)c(s)\big(\int\limits_{\alpha}^{\beta} a(\tau)c(\tau)d\tau\big)^{j-1}.
\end{gather*}
By applying Theorem \ref{thmIntDiffOp} we have for all $x\in D_{[\alpha,\beta]}$ the equality  $ABx=BF(A)x$ holds true if and only if the following equation is fulfilled
\begin{equation}\label{ProofCorIntDiffOpIntDiffEq}
     -a(t)[b(s)c(s)]'=c(s)b(t)a'(t)\sum\limits_{k=1}^{n} \delta_k \big(\int\limits_{\alpha}^{\beta} a(s)c(s)ds\big)^{k-1}.
   \end{equation}
   This  can be rewritten as follows
   \begin{equation}
   \label{EqualityEquivalentABBFAProofCorAdiffBIntOp}
     -a(t)(b(s)c(s))'=c(s)b(t)a'(t)k_1,
   \end{equation}
   where $k_1$ is given by \eqref{constK1ProofCorAdiffBIntOp}, that is, $\di k_1=\sum\limits_{k=1}^{n} \delta_k \big(\int\limits_{\alpha}^{\beta} a(s)c(s)ds\big)^{k-1}$.

   If $k_1=0$ then Equation \eqref{EqualityEquivalentABBFAProofCorAdiffBIntOp} reduces to $a(t)[b(s)c(s)]'=0$. This is equivalent to
   $a(t)=0$ for all $t\in [\alpha,\beta]$ or $b(s)c(s)=\gamma_0$ for some real constant $\gamma_0$. Thus, if there exists $t\in [\alpha,\beta]$ such that $a(t)\not=0$ then $b(s)c(s)=\gamma_0$ for some real constant $\gamma_0$.
   If $k_1\not=0$ then if $t\in \Omega_a $ and $s\in \Omega_c$ then Equation \eqref{EqualityEquivalentABBFAProofCorAdiffBIntOp} can be written as follows:
   \begin{equation*}
     -\frac{(b(s)c(s))'}{c(s)}=\frac{b(t)a'(t)k_1}{a(t)}=\lambda,
   \end{equation*}
   for some real constant $\lambda$. We split this in four cases:
    \begin{itemize}[leftmargin=*]
    \item case 1:
     if $(t,s)\in \Omega_a\times \Omega_c$ then we have
   \begin{equation*}
     b(t)=\frac{a(t)\lambda}{k_1a'(t)},\quad
     c(s)={\rm \exp}\big(\displaystyle\int\limits_{[\xi_0,s]\cap\,\Omega_a\,\cap\, \Omega_c} \frac{-k_1a'(\tau)^2+a(\tau)a''(\tau)-a'(\tau)^2}{a(\tau)a'(\tau)}d\tau\big),
   \end{equation*}
    where $\xi_0\in \Omega_a\cap \Omega_c$.
\item case 2: if $t\in \Omega_a\times([\alpha,\beta]\setminus \Omega_c)$, then Equation \eqref{EqualityEquivalentABBFAProofCorAdiffBIntOp} reduces to
    $a(t)[b(s)c(s)]'=0$. Since $a(t)\not=0$ for  $t\in \Omega_a$, we have $[b(s)c(s)]'=0$. This implies that $c'(s)b(s)=0$ for $s\in [\alpha,\beta]\setminus \Omega_c$, because $c(s)=0$ in this set. By applying Lemma \ref{LemmaContinuousFunctionsSupportMeasureZeroIntOpDiffOP} this is equivalent to  the set
    \begin{gather*}
       {\rm supp }\, c'\ \cap \, {\rm supp }\, b\, \cap\, [\alpha,\beta]\setminus \Omega_c  =\emptyset.
    \end{gather*}
\item case 3: if $(t,s)\in ([\alpha,\beta]\setminus\Omega_a)\times\Omega_c$ then
    either $t\not\in\, {\rm supp}\, a$ or $t\not\in\, {\rm supp}\, a'$. So we split this as follows:
     \begin{itemize}[leftmargin=*]
     \item If $t\not\in\,{\rm supp}\, a$ then
    Equation \eqref{EqualityEquivalentABBFAProofCorAdiffBIntOp} reduces to
    $c(s)b(t)a'(t)k_1=0$. Since $c(s)\not=0$,  $s\in \Omega_c$, we have $b(t)a'(t)k_1=0$. Since $k_1\not=0$, this is equivalent to  $a'(t)b(t)=0$. By applying Lemma \ref{LemmaContinuousFunctionsSupportMeasureZeroIntOpDiffOP}, this is equivalent to  the set
    \begin{gather*}
      {\rm supp }\, a' \, \cap \, {\rm supp}\, b\, \cap\,[\alpha,\beta]\setminus\Omega_a=\emptyset .
    \end{gather*}
\item If $t\not\in \, {\rm supp\,}a'\, $ then Equation \eqref{EqualityEquivalentABBFAProofCorAdiffBIntOp} reduces to the following
    $a(t)[b(s)c(s)]'=0$, for $s\in \Omega_c$. If $t\in  {\rm supp\,}a $ then this equation reduces to $[b(s)c(s)]'=0$, that is, $b(s)c(s)=\gamma_0$ for some real constant $\gamma_0$. Otherwise,
    $t\not\in  {\rm supp\,}a $ therefore the equation $a(t)(b(s)c(s))'=0$ is fulfilled.
    \end{itemize}
\item case 4: if $(t,s)\in ([\alpha,\beta]\setminus\Omega_a)\times ([\alpha,\beta]\setminus \Omega_c)$ then we split this as follows:
       \begin{itemize}[leftmargin=*]
     \item If $t\not\in\,{\rm supp}\, a$ then
    \eqref{EqualityEquivalentABBFAProofCorAdiffBIntOp} reduces to
    $c(s)b(t)a'(t)k_1=0$. Since $c(s)=0$,  $s\in [\alpha,\beta]\setminus\Omega_c$, then  \eqref{EqualityEquivalentABBFAProofCorAdiffBIntOp} is fulfilled.
    \item If $t\not\in \, {\rm supp\,}a'\, $ then Equation \eqref{EqualityEquivalentABBFAProofCorAdiffBIntOp} reduces to the following
    $a(t)[b(s)c(s)]'=0$, for $s\in [\alpha,\beta]\setminus\Omega_c$. If $t\in  {\rm supp\,}a $ then this equation reduces to $[b(s)c(s)]'=0$, that is, $b(s)c'(s)=0$, since $c(s)=0$ in $[\alpha,\beta]\setminus \Omega_c$. By applying Lemma \ref{LemmaContinuousFunctionsSupportMeasureZeroIntOpDiffOP} this is equivalent to the set
    \begin{equation*}
      {\rm supp }\, b \, \cap \, {\rm supp\, } c'\,=\emptyset .
    \end{equation*}
    Otherwise, $t\not\in  {\rm supp\,}a $ and the equation $a(t)[b(s)c(s)]'=0$ is fulfilled.
    \end{itemize}
  \end{itemize}
This completes the proof. \qed
\end{proof}

\begin{ex}{\rm
Consider linear operators $A:C^\infty[\alpha,\beta]\to C^\infty[\alpha,\beta]$,  $B:C^\infty[\alpha,\beta]\to C^\infty[\alpha,\beta]$
defined as follows
\begin{equation*}
  (Ax)(t)= \int\limits_{\alpha}^{\beta} a(t)c(s)x(s)ds,\quad (Bx)(t)= b(t)\frac{dx}{dt},
\end{equation*}
where $\alpha,\beta$ are real numbers, $a(\cdot)$, $b(\cdot)$ and $c(\cdot)$ are real constants. Then these operators satisfy the relation $ABx=BA^nx$
for $n=1,2,\cdots$ and for each $x\in D_{[\alpha,\beta]}$, where $D_{[\alpha,\beta]}$ is given in \eqref{SetDCinfinityXofalphabetazero}. In fact, it follows by Corollary \ref{corCorrespondingthmIntDiffOp1.1}, since in this case the sets
$
 \Omega_a, \ {\rm supp }\, a',\,\ {\rm supp }\, b',\,\ {\rm supp }\, c'
$
are all empty.
}\end{ex}

\begin{ex}{\rm
Consider linear operators $A:C^\infty[\alpha,\beta]\to C^\infty[\alpha,\beta]$,  $B:C^\infty[\alpha,\beta]\to C^\infty[\alpha,\beta]$ be defined as follows
\begin{equation*}
  (Ax)(t)= \int\limits_{\alpha}^{\beta} a(t)c(s)x(s)ds,\quad (Bx)(t)= b(t)\frac{dx}{dt},
\end{equation*}
where $\alpha,\beta$ are real numbers  and $a, b, c:[\alpha,\beta]\to\mathbb{R}$ belong to $C^{\infty}[\alpha,\beta]$ such that
\begin{equation*}
  \int\limits_{\alpha}^{\beta} a(s)c(s)ds=0,\quad {\rm supp}\, b \, \cap \, {\rm supp}\, c=\emptyset.
\end{equation*}
Let $F:\mathbb{R}\to \mathbb{R}$ be defined by $F(t)=\delta_1 t+\delta_2t^2+\ldots+\delta_n t^n$, where
$\delta_1,\ldots,\delta_n $ are constants.
 If $\delta_1=0$ then operators satisfy the commutation relation $  ABx=BF(A)x$ for each $x\in D_{[\alpha,\beta]}$, where $D_{[\alpha,\beta]}$ is given in \eqref{SetDCinfinityXofalphabetazero}. In fact, it follows from Corollary \ref{corCorrespondingthmIntDiffOp1.1}  in the particular case when $k_1=0$ and $b(\cdot)c(\cdot)=0$.
}\end{ex}

\begin{remark}{\rm
In Corollary \ref{corCorrespondingthmIntDiffOp1.1} implicitly there are conditions for $\xi_0$, $k_1$, $\alpha$ and $\beta$ so that we do have representations. Suppose that  functions $a(\cdot)$ and $c(\cdot)$ from operator $A$ in Corollary \ref{corCorrespondingthmIntDiffOp1.1} are such that $\Omega_a=\Omega_c=[\alpha,\beta]$, where $\Omega_a$ and $\Omega_c$ are defined in \eqref{SetOmegaaOmegacProofCorAdiffBIntOp}. Recall that $k_1$ is given by \eqref{constK1ProofCorAdiffBIntOp}, that is,
\begin{equation*}
  k_1=\sum\limits_{k=1}^{n} \delta_k \left(\int\limits_{\alpha}^{\beta} a(s)c(s)ds\right)^{k-1},
\end{equation*}
where $\delta_k\in\mathbb{R}$, $k$ is integer such that $1\le k\le n$ are coefficients of the polynomial $F(t)= \sum\limits_{k=0}^{n}\delta_k t^k$.
If $k_1\not=0$, then $c(s)={\rm exp}\left(\di\int\limits_{\xi_0}^s \frac{-k_1a'(\tau)^2+a(\tau)a''(\tau)-a'(\tau)^2}{a(\tau)a'(\tau)}d\tau\right)$.
By replacing this in \eqref{constK1ProofCorAdiffBIntOp} we get
\begin{equation}\label{EqConditioningCsi0K1AlphaBeta}
  k_1=\sum\limits_{k=1}^{n} \delta_k \left(\int\limits_{\alpha}^{\beta} a(s) {\rm exp}\left(\di\int\limits_{\xi_0}^s \frac{-k_1a'(\tau)^2+a(\tau)a''(\tau)-a'(\tau)^2}{a(\tau)a'(\tau)}d\tau\right) ds\right)^{k-1},
\end{equation}
which gives conditions to find $\xi_0$ given $k_1$, $a(\cdot)$, $\alpha$, $\beta$ and we add the condition  $\alpha\le \xi_0\le \beta$. We consider the following particular cases:
\begin{enumerate}[leftmargin=*, label=\textup{\arabic*.}, ref=\arabic*]
  \item if $a(t)=\nu_0 t^m$, where $\nu_0 \in \mathbb{R}\setminus\{0\}$ and $m\ge 2$ is a positive integer, 
      from \eqref{EqConditioningCsi0K1AlphaBeta} we get
      \begin{gather}\label{EqSubCalculationCsi0GivenK1}
        k_1=\delta_1+\sum_{k=2}^{n} \delta_k \cdot \nu_0^{k-1}\cdot \xi_0^{(1+k_1m)(k-1)}\cdot \left(\int\limits_{\alpha}^{\beta} s^{m-1-k_1m}ds \right)^{k-1}.
      \end{gather}
      We consider two cases to compute the integral:
     \begin{enumerate}[leftmargin=*, label=\textup{\alph*)}]
     \item if $k_1\not=1$ then Equation \eqref{EqSubCalculationCsi0GivenK1} becomes
     \begin{equation*}
        k_1=\delta_1+\sum_{k=2}^{n} \delta_k \cdot \nu_0^{k-1}\cdot \xi_0^{(1+k_1m)(k-1)}\cdot \left(\frac{\beta^{m(1-k_1)}-\alpha^{m(1-k_1)}}{m(1-k_1)} \right)^{k-1},
     \end{equation*}
     where $\alpha\le \xi_0 \le \beta$;

     \item if $k_1=1$ then Equation \eqref{EqSubCalculationCsi0GivenK1} becomes
      \begin{equation*}
        1-\delta_1= \sum_{k=2}^{n} \delta_k \cdot \nu_0^{k-1}\cdot \xi_0^{(1+m)(k-1)}\cdot \left( \ln\left|\frac{\beta}{\alpha}\right|\right)^{k-1},\ \alpha\not=0,\quad \alpha\le \xi_0 \le \beta;
      \end{equation*}
     \end{enumerate}

     \item if $a(t)=\nu_0+\nu_1 t$, where $\nu_0,\, \nu_1 \in \mathbb{R}$, $\nu_1\not=0$ 
        then from \eqref{EqConditioningCsi0K1AlphaBeta} we get
         \begin{equation*}
           k_1=\delta_1+\sum_{k=2}^{n} \delta_k \cdot (\nu_0+\nu_1 \xi_0)^{(1+k_1)(k-1)}\cdot \left(\int\limits_{\alpha}^{\beta} (\nu_0+\nu_1 s)^{-k_1}ds\right)^{k-1}.
         \end{equation*}
         We consider the following cases:
         \begin{enumerate}[leftmargin=*, label=\textup{\alph*)}]
           \item if $k_1\not=1$ we have
           \begin{equation*}\hspace{-0.2cm}
           \sum_{k=2}^{n} \frac{\delta_k\cdot (\nu_0+\nu_1 \xi_0)^{(1+k_1)(k-1)}}{(\nu_1(1-k_1))^{k-1}} \cdot[(\nu_0+\nu_1\beta)^{1-k_1}-(\nu_0+\nu_1 \alpha)^{1-k_1}]^{k-1}
           =k_1-\delta_1,
         \end{equation*}
           where $\alpha\le \xi_0 \le \beta$;

          \item if $k_1=1$ we have
           \begin{equation*}
           1=\delta_1+\sum_{k=2}^{n} \frac{\delta_k}{\nu_1^{k-1}} \cdot (\nu_0+\nu_1 \xi_0)^{2(k-1)}\cdot \left(\ln\left|\frac{\nu_0+\nu_1\beta}{\nu_0+\nu_1\alpha}\right|\right)^{k-1},
         \end{equation*}
         where $\alpha\le \xi_0 \le \beta$.
         \end{enumerate}
\end{enumerate}
}\end{remark}

\begin{remark}{\rm
From Corollary \ref{corCorrespondingthmIntDiffOp1.1} one can construct representations of the covariant type commutation relation $AB=BF(A)$ where the operators $A$ and $B$ are  given in \eqref{IntegralOperatorDegeneratedKernelDiffOp} and the operator $B$ is known. This is done by writing the unknown
kernel of operator $A$ as function of the corresponding coefficient of operator $B$. Moreover, it is also possible to find representations of this commutation relation when one of the functions $a$, $b$ or $c$, which are defined in \eqref{IntegralOperatorDegeneratedKernelDiffOp}, is given.
}\end{remark}



\begin{cor}\label{corCorrespondingthmIntDiffOpCaseOfMonomial}
Let $A:C^\infty[\alpha,\beta]\to C^\infty[\alpha,\beta]$ and  $B:C^\infty[\alpha,\beta]\to C^\infty[\alpha,\beta]$ be linear operators defined as follows
\begin{equation*}
  (Ax)(t)= \int\limits_{\alpha}^{\beta} a(t)c(s)x(s)ds,\quad (Bx)(t)= b(t)\frac{dx}{dt},
\end{equation*}
where $\alpha,\beta\in\mathbb{R}$ and $a, b, c:[\alpha,\beta]\to\mathbb{R}$ belong to $C^\infty[\alpha,\beta]$. Set
\begin{gather}\nonumber
  \Omega_a=\{t\in[\alpha,\beta]:\, a(t)\not=0,\ a'(t)\not=0\},\quad \Omega_c=\{s\in[\alpha,\beta]:\, c(s)\not=0\},  \\
  \nonumber
  Q_{[\alpha,\beta]}(a,c)= \int\limits_{\alpha}^{\beta} a(s)c(s)ds. 
\end{gather}
 Suppose that the function $a$ is given and $\Omega_a=\Omega_c$ with exception perhaps of countably many points in $[\alpha,\beta]$. Then
   for all $x\in D_{[\alpha,\beta]}$, where $D_{[\alpha,\beta]}$ is given in \eqref{SetDCinfinityXofalphabetazero}, that is, $D_{[\alpha,\beta]}=\{x\in C^\infty[\alpha,\beta]:\ x(\alpha)=x(\beta)=0\}$, for some positive integer $n$ and   some non-zero real constant $\delta$ the equality $ABx=\delta BA^nx$ holds true
if and only if one of the following is true
    \begin{enumerate}[leftmargin=*, label=\textup{\arabic*.}, ref=\arabic*]
        \item if $Q_{[\alpha,\beta]}(a,c)=0$ and $n\not=1$ then if ${\rm supp}\, a\not=\emptyset$ then $b(s)c(s)=\gamma_0$ for some real constant $\gamma_0$;
        \item if $Q_{[\alpha,\beta]}(a,c)\not=0$, $(t,s)\in \Omega_a\times \Omega_c$  then the following conditions are fulfilled:
     \begin{enumerate}[leftmargin=*, label=\textup{\alph*)}]
         \item
     if $(t,s)\in \Omega_a\times \Omega_c$ then
   $\resizebox{0.35\hsize}{!}{$ \displaystyle
     b(t)=\frac{a(t)\lambda}{\delta Q_{[\alpha,\beta]}(a,c)^{n-1} a'(t)}$}$, \\
    $ \resizebox{0.9\hsize}{!}{$ \displaystyle \hspace{-4mm} c(s)=c_{\xi_0}(s)={\rm exp}\left(\di\int\limits_{[\xi_0,s]\cap\,\Omega_a\,\cap\, \Omega_c} \hspace{-6mm}\frac{- \delta Q_{[\alpha,\beta]}(a,c)^{n-1} a'(\tau)^2+a(\tau)a''(\tau)-a'(\tau)^2}{a(\tau)a'(\tau)}d\tau\right), $}$

     for some  non-zero real constant $\lambda$ and $\xi_0\in \Omega_a\cap \Omega_c$;

       %


    \item if $(t,s)\in \Omega_a\times ([\alpha,\beta]\setminus \Omega_c)$ then  the set
    \begin{equation*}
       {\rm supp }\, c'\ \cap \, {\rm supp}\, b\, \cap\, [\alpha,\beta]\setminus \Omega_c  =\emptyset;
    \end{equation*}

    \item if $(t,s)\in ([\alpha,\beta]\setminus\Omega_a)\times \Omega_c$ then either the set
    \begin{equation*}
      {\rm supp }\, a' \, \cap \, {\rm supp}\, b\, \cap\,[\alpha,\beta]\setminus\Omega_a=\emptyset
    \end{equation*}
    or ${\rm supp }\,a'\, \setminus\, {\rm supp}\, a\not=\emptyset$ implies $b(s)c(s)=\gamma_0$ for some real constant $\gamma_0$;

    \item if $(t,s)\in ([\alpha,\beta]\setminus\Omega_a)\times ([\alpha,\beta]\setminus\Omega_c)$ then
          ${\rm supp }\,a'\, \setminus\, {\rm supp}\, a\not=\emptyset$ implies that the set
          \begin{equation*}
             {\rm supp }\, c' \, \cap \, {\rm supp}\, b\, \cap\,[\alpha,\beta]\setminus\Omega_c=\emptyset.
          \end{equation*}
   \end{enumerate}
  \end{enumerate}
\end{cor}

\begin{proof}
  This follows by Corollary \ref{corCorrespondingthmIntDiffOp1.1} when
  $$
  k_1=\left\{\begin{array}{rr}
  \delta, & \mbox{ if }\ n=1 \\
  \delta Q_{[\alpha,\beta]}(a,c)^{n-1}, & \mbox{ if } n\not=1.
  \end{array}\right.
  $$
  Since $\delta$ is a non-zero constant and $n\in\mathbb{N}$, then $k_1=0$ if and only if $Q_{[\alpha,\beta]}(a,c)=0$ and $n\not=1$.
\qed\end{proof}

\begin{example}{\rm
Consider linear operators $A:C^\infty[\alpha,\beta]\to C^\infty[\alpha,\beta]$,  $B:C^\infty[\alpha,\beta]\to C^\infty[\alpha,\beta]$ be defined as follows
\begin{equation*}
  (Ax)(t)= \int\limits_{\alpha}^{\beta} a(t)c(s)x(s)ds,\quad (Bx)(t)= b(t)\frac{dx}{dt},
\end{equation*}
where $\alpha,\beta$ are real numbers  and $a, b, c:[\alpha,\beta]\to\mathbb{R}$ belong to $C^{\infty}[\alpha,\beta]$ such that
\begin{equation*}
  \int\limits_{\alpha}^{\beta} a(s)c(s)ds=0,\quad {\rm supp}\, b \, \cap \, {\rm supp}\, c=\emptyset.
\end{equation*}
Let $F:\mathbb{R}\to \mathbb{R}$ be defined by $F(t)=\delta_1t+\delta_2t^2+\ldots+\delta_n t^n$, where
$\delta_1,\ldots,\delta_n $ are constants.
 If $\delta_1=0$ then operators satisfy the commutation relation $  ABx=BF(A)x$ for each $x\in D_{[\alpha,\beta]}$, where $D_{[\alpha,\beta]}$ is given in \eqref{SetDCinfinityXofalphabetazero}. In fact, it follows from Corollary \ref{corCorrespondingthmIntDiffOp1.1}  in the particular case when $k_1=0$ and $b(\cdot)c(\cdot)=0$.
}\end{example}

\begin{remark}{\rm
In Corollary \ref{corCorrespondingthmIntDiffOpCaseOfMonomial} implicitly there are conditions for $\xi_0$, $Q_{[\alpha,\beta]}(a,c)$, $\alpha$ and $\beta$ so that we do have representations for the commutation relation $ABx=\delta BA^nx$, for each $x\in D_{[\alpha,\beta]}$ and some $\delta\in\mathbb{R}\setminus\{0\}$. Suppose that  functions $a(\cdot)$ and $c(\cdot)$ from operator $A$ in Corollary \ref{corCorrespondingthmIntDiffOpCaseOfMonomial} are such that $\Omega_a=\Omega_c=[\alpha,\beta]$, where $\Omega_a$ and $\Omega_c$ are defined in \eqref{SetOmegaaOmegacProofCorAdiffBIntOp}. Recall that $Q_{[\alpha,\beta]}(a,c)$ is given by
\begin{equation}\label{DefofMiuRemCorThmIntDiffOpMonomial}
  Q_{[\alpha,\beta]}(a,c)=\int\limits_{\alpha}^{\beta} a(s)c(s)ds.
\end{equation}
If $Q_{[\alpha,\beta]}(a,c)\not=0$ then
\begin{equation*}
c(s)={\rm exp}\left(\di\int\limits_{\xi_0}^s \frac{-\delta Q_{[\alpha,\beta]}(a,c)^{n-1} a'(\tau)^2+a(\tau)a''(\tau)-a'(\tau)^2}{a(\tau)a'(\tau)}d\tau\right).
\end{equation*}
By replacing this in \eqref{DefofMiuRemCorThmIntDiffOpMonomial} we get
\begin{equation}\label{EqConditioningCsi0AlphaBetaMuMonomial}
 \resizebox{0.95\hsize}{!}{$\displaystyle Q_{[\alpha,\beta]}(a,c)=\int\limits_{\alpha}^{\beta} a(s) {\rm exp}\big(\di\int\limits_{\xi_0}^s \frac{-\delta Q_{[\alpha,\beta]}(a,c)^{n-1}a'(\tau)^2+a(\tau)a''(\tau)-a'(\tau)^2}{a(\tau)a'(\tau)}d\tau\big) ds, $}
\end{equation}
which gives conditions to find $\xi_0$ given $Q_{[\alpha,\beta]}(a,c)$, $a(\cdot)$, $\alpha$, $\beta$ and we add the condition  $\alpha\le \xi_0\le \beta$. We consider the following particular cases:
\begin{enumerate}[leftmargin=*, label=\textup{\arabic*.}, ref=\arabic*]
  \item if $a(t)=\nu_0 t^m$, where $\nu_0 \in \mathbb{R}\setminus\{0\}$ and $m\ge 2$ is a positive integer, 
      from \eqref{EqConditioningCsi0AlphaBetaMuMonomial} we get
      \begin{equation}\label{EqSubCalculationCsi0GivenMiu}
        Q_{[\alpha,\beta]}(a,c)=\nu_0\cdot \xi_0^{(1+\delta Q_{[\alpha,\beta]}(a,c)^{n-1} m)}\cdot \int\limits_{\alpha}^{\beta} s^{m-1-\delta Q_{[\alpha,\beta]}(a,c)^{n-1} m}ds .
      \end{equation}
      We consider the following cases:
      \begin{enumerate}[leftmargin=*, label=\textup{\alph*)}]
     \item if $\delta Q_{[\alpha,\beta]}(a,c)^{n-1} \not=1$ then Equation \eqref{EqSubCalculationCsi0GivenMiu} becomes
     \begin{eqnarray}\nonumber
   &&     Q_{[\alpha,\beta]}(a,c) m(1-\delta Q_{[\alpha,\beta]}(a,c)^{n-1})= \nu_0\cdot \xi_0^{(1+\delta Q_{[\alpha,\beta]}(a,c)^{n-1}m)}
   \\
   && \label{EqSubCalculationCsi0GivenMiuIntegralNotLog}
       \cdot  \left[{\beta^{m(1-\delta Q_{[\alpha,\beta]}(a,c)^{n-1})}-\alpha^{m(1-\delta Q_{[\alpha,\beta]}(a,c)^{n-1})}}\right],
     \end{eqnarray}
     where $\alpha\le \xi_0 \le \beta$. We consider the following cases:

         \noindent if $1+\delta m Q_{[\alpha,\beta]}(a,c)^{n-1}\not=0$, $\xi_0>0$,
          \begin{equation*}
         \frac{Q_{[\alpha,\beta]}(a,c)}{\nu_0}\cdot \frac{Q_{[\alpha,\beta]}(a,c) m(1-\delta Q_{[\alpha,\beta]}(a,c)^{n-1})}{{\beta^{m(1-\delta Q_{[\alpha,\beta]}(a,c)^{n-1})}-\alpha^{m(1-\delta Q_{[\alpha,\beta]}(a,c)^{n-1})}}}>0
         \end{equation*}
        then \eqref{EqSubCalculationCsi0GivenMiuIntegralNotLog} has a real solution $\xi_0$ given by
         \begin{eqnarray*}
         &&  \xi_0={\rm exp}\left(\frac{1}{1+\delta m Q_{[\alpha,\beta]}(a,c)^{n-1}}\cdot \ln\left[\frac{Q_{[\alpha,\beta]}(a,c)}{\nu_0}\right.\right.\cdot  \\
         &&  \cdot \left.\left.\frac{m(1-\delta Q_{[\alpha,\beta]}(a,c)^{n-1})}{\beta^{m(1-\delta Q_{[\alpha,\beta]}(a,c)^{n-1})}-\alpha^{m(1-\delta Q_{[\alpha,\beta]}(a,c)^{n-1})}}\right] \right);
         \end{eqnarray*}

         \noindent if $1+\delta m Q_{[\alpha,\beta]}(a,c)^{n-1}\not=0$, $\xi_0>0$,
          \begin{equation*}
         \frac{Q_{[\alpha,\beta]}(a,c)}{\nu_0}\cdot \frac{Q_{[\alpha,\beta]}(a,c) m(1-\delta Q_{[\alpha,\beta]}(a,c)^{n-1})}{{\beta^{m(1-\delta Q_{[\alpha,\beta]}(a,c)^{n-1})}-\alpha^{m(1-\delta Q_{[\alpha,\beta]}(a,c)^{n-1})}}}<0
         \end{equation*}
        then \eqref{EqSubCalculationCsi0GivenMiuIntegralNotLog} has no real solution.

         \noindent if $1+\delta m Q_{[\alpha,\beta]}(a,c)^{n-1}=0$, $\xi_0\ge 0$ then  $\xi_0^{1+\delta m Q_{[\alpha,\beta]}(a,c)^{n-1} }=1$, and so,  Equation \eqref{EqSubCalculationCsi0GivenMiuIntegralNotLog} has a real solution when
             \begin{equation*}
               Q_{[\alpha,\beta]}(a,c)=\nu_0\cdot \frac{\beta^{m+1}-\alpha^{m+1}}{m+1};
             \end{equation*}

          \noindent if $1+\delta m Q_{[\alpha,\beta]}(a,c)^{n-1}>0$, $\xi_0= 0$ then  $\xi_0^{1+\delta m Q_{[\alpha,\beta]}(a,c)^{n-1} }=1$, and so,  \eqref{EqSubCalculationCsi0GivenMiuIntegralNotLog} has no real solution since it implies $Q_{[\alpha,\beta]}(a,c)=0$, but $Q_{[\alpha,\beta]}(a,c)\not=0$ by assumption
              of the Corollary \ref{corCorrespondingthmIntDiffOpCaseOfMonomial}.

          \noindent if $\xi_0< 0$, $1+\delta m Q_{[\alpha,\beta]}(a,c)^{n-1}=\frac{l_1}{l_2}$, where $l_1,l_2\in \mathbb{Z}$, $l_2\not=0$, $l_1/l_2$ is an irreducible fraction, $\di \frac{Q_{[\alpha,\beta]}(a,c)}{\nu_0}\cdot  \frac{m(1-\delta Q_{[\alpha,\beta]}(a,c)^{n-1})}{\beta^{m(1-\delta Q_{[\alpha,\beta]}(a,c)^{n-1})}-\alpha^{m(1-\delta Q_{[\alpha,\beta]}(a,c)^{n-1})}}<0$ then
                \begin{itemize}[leftmargin=*]
                  \item if either $l_1$ or $l_2$ is even then there is no real solution for Equation \eqref{EqSubCalculationCsi0GivenMiuIntegralNotLog};
                  \item if $l_1, l_2$ are odd then Equation \eqref{EqSubCalculationCsi0GivenMiuIntegralNotLog} has a real solution $\xi_0$ given by
                    \begin{equation*}
                      \xi_0=\sqrt[l_1]{\left(\frac{Q_{[\alpha,\beta]}(a,c)}{\nu_0}\cdot \frac{m(1-\delta Q_{[\alpha,\beta]}(a,c)^{n-1})}{\beta^{m(1-\delta Q_{[\alpha,\beta]}(a,c)^{n-1})}-\alpha^{m(1-\delta Q_{[\alpha,\beta]}(a,c)^{n-1})}} \right)^{l_2}};
                    \end{equation*}
                 \end{itemize}
         \noindent if $\xi_0< 0$, $1+\delta m Q_{[\alpha,\beta]}(a,c)^{n-1}=\frac{l_1}{l_2}$, where $l_1,l_2\in \mathbb{Z}$, $l_2\not=0$, $l_1/l_2$ is an irreducible fraction, $\di \frac{Q_{[\alpha,\beta]}(a,c)}{\nu_0}\cdot \frac{m(1-\delta Q_{[\alpha,\beta]}(a,c)^{n-1})}{\beta^{m(1-\delta Q_{[\alpha,\beta]}(a,c)^{n-1})}-\alpha^{m(1-\delta Q_{[\alpha,\beta]}(a,c)^{n-1})}}>0$ then
             \begin{itemize}[leftmargin=*]
                       \item if either $l_1$, $l_2$ are odd or $l_2$ is even then there is no real solution for Equation \eqref{EqSubCalculationCsi0GivenMiuIntegralNotLog};
                  \item if $l_1$ is even and $l_2$ is odd then Equation \eqref{EqSubCalculationCsi0GivenMiuIntegralNotLog} has a real solution $\xi_0$ given by
                    \begin{eqnarray*}
                  &&    \xi_0=-\sqrt[l_1]{\left(\frac{Q_{[\alpha,\beta]}(a,c)}{\nu_0}\cdot m(1-\delta Q_{[\alpha,\beta]}(a,c)^{n-1})\right)^{l_2}}  \\
                  && \cdot
                     \sqrt[l_1]{ \left(\frac{1}{\beta^{m(1-\delta Q_{[\alpha,\beta]}(a,c)^{n-1})}-\alpha^{m(1-\delta Q_{[\alpha,\beta]}(a,c)^{n-1})}} \right)^{l_2}};
                    \end{eqnarray*}
                \end{itemize}

               \noindent if $\xi_0< 0$ and $1+\delta m Q_{[\alpha,\beta]}(a,c)^{n-1}$ is an irrational number then \\ $\xi_0^{1+\delta m Q_{[\alpha,\beta]}(a,c)^{n-1}}$ is not defined, thus Equation \eqref{EqSubCalculationCsi0GivenMiuIntegralNotLog} has no real solution.

     \item if $\delta Q_{[\alpha,\beta]}(a,c)^{n-1}=1$ then Equation \eqref{EqSubCalculationCsi0GivenMiu} becomes
      \begin{equation}\label{EqSubCalculationCsi0GivenMiuIntegralLog}
        Q_{[\alpha,\beta]}(a,c)= \nu_0 \xi_0^{1+m}\cdot \ln\left|\frac{\beta}{\alpha}\right|,\ \alpha\not=0, \, \beta\not=0,\,\ln\left|\frac{\beta}{\alpha}\right|\not=0  \quad \alpha\le \xi_0 \le \beta.
      \end{equation}
      We consider the following cases:

          \noindent if $m$ is even then  Equation \eqref{EqSubCalculationCsi0GivenMiuIntegralLog} has a real solution $\xi_0$ given by
          \begin{equation*}
            \xi_0=\sqrt[1+m]{\frac{Q_{[\alpha,\beta]}(a,c)}{\nu_0 \cdot \ln\left|\frac{\beta}{\alpha}\right|}},\ \alpha\not=0,\ \beta\not=0,\,\ln\left|\frac{\beta}{\alpha}\right|\not=0.
          \end{equation*}

          \noindent if $m$ is odd then  Equation \eqref{EqSubCalculationCsi0GivenMiuIntegralLog} has a real solution if and only if
           $\frac{Q_{[\alpha,\beta]}(a,c)}{\nu_0\cdot \ln\left|\frac{\beta}{\alpha}\right|}>0$, $\alpha\not=0$ which is
           given by
           \begin{equation*}
            \xi_0=\pm \sqrt[1+m]{\frac{Q_{[\alpha,\beta]}(a,c)}{\nu_0 \cdot \ln\left|\frac{\beta}{\alpha}\right|}},\ \alpha\not=0,\ \beta\not=0,\,\ln\left|\frac{\beta}{\alpha}\right|\not=0.
          \end{equation*}
        \end{enumerate}

     \item if $a(t)=\nu_0+\nu_1 t$, where $\nu_0,\, \nu_1 \in \mathbb{R}$, $\nu_1\not=0$ 
        then from \eqref{EqConditioningCsi0AlphaBetaMuMonomial} we get
         \begin{equation*}
          Q_{[\alpha,\beta]}(a,c)=(\nu_0+\nu_1 \xi_0)^{1+\delta Q_{[\alpha,\beta]}(a,c)^{n-1}}\cdot \int\limits_{\alpha}^{\beta} (\nu_0+\nu_1 s)^{-\delta Q_{[\alpha,\beta]}(a,c)^{n-1}}ds.
         \end{equation*}
         We consider the following cases:
         \begin{enumerate}[leftmargin=*, label=\textup{\alph*)}]
           \item if $\delta Q_{[\alpha,\beta]}(a,c)^{n-1}\not=1$ we have
           \begin{eqnarray}\nonumber
        &&   (\nu_0+\nu_1 \xi_0)^{1+\delta Q_{[\alpha,\beta]}(a,c)^{n-1}}\\ \label{EqRemarkXiZeroAffineFunctionIntegralNotLog}
         &&   =\frac{1}{\nu_1}\frac{Q_{[\alpha,\beta]}(a,c) (1-\delta Q_{[\alpha,\beta]}(a,c)^{n-1})}{ (\nu_0+\nu_1 \beta)^{1-\delta Q_{[\alpha,\beta]}(a,c)^{n-1}} - (\nu_0+\nu_1 \alpha)^{1-\delta Q_{[\alpha,\beta]}(a,c)^{n-1}}},
         \end{eqnarray}
           where $\alpha\le \xi_0 \le \beta$. We consider the following cases:

             \noindent if $\nu_0+\nu_1 \xi_0>0$, $1+\delta Q_{[\alpha,\beta]}(a,c)^{n-1}\not=0$ and
          \begin{equation*}
             \frac{1}{\nu_1}\cdot \frac{Q_{[\alpha,\beta]}(a,c)(1-\delta Q_{[\alpha,\beta]}(a,c)^{n-1})}{(\nu_0+\nu_1\beta)^{1-\delta Q_{[\alpha,\beta]}(a,c)^{n-1}}-(\nu_0+\nu_1\alpha)^{1-\delta Q_{[\alpha,\beta]}(a,c)^{n-1}}}>0
           \end{equation*}
              then
                  Equation \eqref{EqRemarkXiZeroAffineFunctionIntegralNotLog} has solution $\xi_0$ given by
                  \begin{eqnarray*}
                    \xi_0 &=&-\frac{\nu_0}{\nu_1} +\frac{1}{\nu_1}{\rm exp}\left(\frac{1}{1+\delta Q_{[\alpha,\beta]}(a,c)^{n-1}}\cdot \right. \\
                  & & \left.  \cdot\ln\frac{1}{\nu_1}\cdot\frac{Q_{[\alpha,\beta]}(a,c)(1-\delta Q_{[\alpha,\beta]}(a,c)^{n-1})}{(\nu_0+\nu_1\beta)^{1-\delta Q_{[\alpha,\beta]}(a,c)^{n-1}}-(\nu_0+\nu_1 \alpha)^{1-\delta Q_{[\alpha,\beta]}(a,c)^{n-1}}} \right);
                  \end{eqnarray*}

              \noindent if $\nu_0+\nu_1 \xi_0>0$, $1+\delta Q_{[\alpha,\beta]}(a,c)^{n-1}\not=0$ and
                \begin{equation*}
              \frac{1}{\nu_1}\cdot \frac{Q_{[\alpha,\beta]}(a,c)(1-\delta Q_{[\alpha,\beta]}(a,c)^{n-1})}{(\nu_0+\nu_1\beta)^{1-\delta Q_{[\alpha,\beta]}(a,c)^{n-1}}-(\nu_0+\nu_1\alpha)^{1-\delta Q_{[\alpha,\beta]}(a,c)^{n-1}}}<0
                \end{equation*}
               then
                  Equation \eqref{EqRemarkXiZeroAffineFunctionIntegralNotLog} has no real solution;

              \noindent if $\nu_0+\nu_1 \xi_0 \ge 0$, $1+\delta Q_{[\alpha,\beta]}(a,c)^{n-1}=0$ then $(\nu_0+\nu_1 \xi_0)^{1+\delta Q_{[\alpha,\beta]}(a,c)^{n-1}}=1$ hence
                  Equation \eqref{EqRemarkXiZeroAffineFunctionIntegralNotLog} has solution a real $\xi_0=-\frac{\nu_0}{\nu_1}$ when
                  \begin{equation*}
                    Q_{[\alpha,\beta]}(a,c)=\frac{(\nu_0+\nu_1\beta)^2-(\nu_0+\nu_1\alpha)^2}{2};
                  \end{equation*}

               \noindent if $\nu_0+\nu_1 \xi_0=0$, $1+\delta Q_{[\alpha,\beta]}(a,c)^{n-1}>0$ and
               \begin{equation*}
               \frac{1}{\nu_1}\cdot \frac{Q_{[\alpha,\beta]}(a,c)(1-\delta Q_{[\alpha,\beta]}(a,c)^{n-1})}{(\nu_0+\nu_1\beta)^{1-\delta Q_{[\alpha,\beta]}(a,c)^{n-1}}-(\nu_0+\nu_1\alpha)^{1-\delta Q_{[\alpha,\beta]}(a,c)^{n-1}}}>0
              \end{equation*}
                then $(\nu_0+\nu_1 \xi_0)^{1+\delta Q_{[\alpha,\beta]}(a,c)^{n-1}}=0$ and so
                  Equation \eqref{EqRemarkXiZeroAffineFunctionIntegralNotLog} has no real solution since this implies $Q_{[\alpha,\beta]}(a,c)=0$ which contradicts the hypothesis of Corollary \ref{corCorrespondingthmIntDiffOpCaseOfMonomial};

                \noindent if $\nu_0+\nu_1\xi_0< 0$, $1+\delta  Q_{[\alpha,\beta]}(a,c)^{n-1}=\frac{l_1}{l_2}$, where $l_1,l_2\in \mathbb{Z}$, $l_2\not=0$, $l_1/l_2$ is an irreducible fraction,
                \begin{equation*}
                \di \frac{Q_{[\alpha,\beta]}(a,c)}{\nu_0}\cdot \frac{1-\delta Q_{[\alpha,\beta]}(a,c)^{n-1}}{(\nu_0+\nu_1\beta)^{1-\delta Q_{[\alpha,\beta]}(a,c)^{n-1}}-(\nu_0+\nu_1\alpha)^{1-\delta Q_{[\alpha,\beta]}(a,c)^{n-1}}}>0
                \end{equation*}
                then
                  \begin{itemize}[leftmargin=*]
                       \item if either $l_1$, $l_2$ are odd or $l_2$ is even then there is no real solution for Equation \eqref{EqRemarkXiZeroAffineFunctionIntegralNotLog};
                  \item if $l_1$ is even and $l_2$ is odd then Equation \eqref{EqRemarkXiZeroAffineFunctionIntegralNotLog} has a real solution $\xi_0$ given by
                    \begin{eqnarray*}
                      \xi_0 &=&-\frac{\nu_0}{\nu_1}-
                      \frac{1}{\nu_1}\sqrt[l_1]{\left(\frac{Q_{[\alpha,\beta]}(a,c)}{\nu_0}\right)^{l_2}} \\ & &\cdot \sqrt[l_1]{\left(\frac{1-\delta Q_{[\alpha,\beta]}(a,c)^{n-1}}{(\nu_0+\nu_1\beta)^{1-\delta Q_{[\alpha,\beta]}(a,c)^{n-1}}-(\nu_0+\nu_1\alpha)^{1-\delta Q_{[\alpha,\beta]}(a,c)^{n-1}}} \right)^{l_2}};
                    \end{eqnarray*}
                \end{itemize}

               \noindent if $\nu_0+\nu_1\xi_0< 0$ and $1+\delta  Q_{[\alpha,\beta]}(a,c)^{n-1}$ is an irrational number then the expression $(\nu_0+\nu_1\xi_0)^{1+\delta  Q_{[\alpha,\beta]}(a,c)^{n-1}}$ is not defined, thus Equation \eqref{EqRemarkXiZeroAffineFunctionIntegralNotLog} has no real solution.

          \item if $\delta Q_{[\alpha,\beta]}(a,c)^{n-1}=1$ we have
           \begin{equation}\label{EqRemarkXiZeroCondAffineFunctionIntegralLog}
           (\nu_0+\nu_1\xi_0)^2=\frac{Q_{[\alpha,\beta]}(a,c)\cdot \nu_1}{\ln\left|\frac{\nu_0+\nu_1 \beta}{\nu_0+\nu_1 \alpha}\right|},
         \end{equation}
         where $\alpha\le \xi_0 \le \beta$. We consider the following cases:

                \noindent if $\frac{Q_{[\alpha,\beta]}(a,c)\cdot \nu_1}{\ln\left|\frac{\nu_0+\nu_1 \beta}{\nu_0+\nu_1 \alpha}\right|}\ge 0 $ then
                Equation \eqref{EqRemarkXiZeroCondAffineFunctionIntegralLog} has real solutions $\xi_0$ given by
                \begin{equation*}
                  \xi_0=-\frac{\nu_0}{\nu_1}\pm \frac{1}{\nu_1}\sqrt{\frac{Q_{[\alpha,\beta]}(a,c)\cdot \nu_1}{\ln\left|\frac{\nu_0+\nu_1 \beta}{\nu_0+\nu_1 \alpha}\right|}};
                \end{equation*}

                \noindent if $\frac{Q_{[\alpha,\beta]}(a,c)\cdot \nu_1}{\ln\left|\frac{\nu_0+\nu_1 \beta}{\nu_0+\nu_1 \alpha}\right|}< 0 $ then
                Equation \eqref{EqRemarkXiZeroCondAffineFunctionIntegralLog} has no real solution.

         \end{enumerate}
\end{enumerate}
}\end{remark}


\begin{ex}{\rm
Consider the set
  \begin{gather}\label{SubsetOfCinfinityZeroIntheboundary01}
  D_{[0,1]}=\left\{ x\in C^{\infty}[0,1]:\ x(0)=x(1)=0\right\}.
  \end{gather}
We seek for a particular solution from the following relations
  \begin{eqnarray*}
     b(t)&=&\frac{a(t)\lambda}{\delta Q_{[\alpha,\beta]}(a,c)^{n-1} a'(t)},\\
     c(s)&=&{\rm exp}\left(\di\int\limits_{[\xi_0,s]\cap\,\Omega_a\,\cap\, \Omega_c} \frac{- \delta Q_{[\alpha,\beta]}(a,c)^{n-1} a'(\tau)^2+a(\tau)a''(\tau)-a'(\tau)^2}{a(\tau)a'(\tau)}d\tau\right),
   \end{eqnarray*}
   for some  $\lambda\in\mathbb{R}\setminus\{0\}$, $\xi_0\in \Omega_a\cap \Omega_c$, $\Omega_a$ and $\Omega_c$ are defined in
   \eqref{SetOmegaaOmegacProofCorAdiffBIntOp}, that is,
  $$
  \Omega_a=\{t\in[\alpha,\beta]:\, a(t)\not=0,\ a'(t)\not=0\}, \quad   \Omega_c=\{s\in[\alpha,\beta]:\, c(s)\not=0\}.
   $$
  We regard $b(\cdot)=-a(\cdot)$ and look for $c$ such that $Q_{[\alpha,\beta]}(a,c)=\int\limits_{\alpha}^{\beta} a(s)c(s)ds=1$, $\alpha=0$, $\beta=1$ and $\delta=1$. Then, we get the following:
  \begin{gather*}
    a(t)=-b(t)=(t+1)\ln 2 ,\quad c(t)=\frac{1}{(\ln 2)^2(t+1)^2},\quad t\in\, [0,1],\quad \xi_0=\frac{1-\ln 2}{\ln 2}.
  \end{gather*}
  Thus, by applying Corollary \ref{corCorrespondingthmIntDiffOpCaseOfMonomial}, we have that operators $A:\, C^\infty[0,1]\to C^\infty[0,1]$ and $B:\, C^\infty[0,1]\to C^\infty[0,1]$ defined as follows
  \begin{gather*}
    (Ax)(t)= \int\limits_{0}^{1}  \frac{1}{\ln 2}\cdot\frac{t+1}{ (s+1)^2 }x(s)ds, \quad (Bx)(t)=-\ln 2 (t+1)\frac{dx}{dt}
  \end{gather*}
  satisfy the relation $ABx=BA^nx$, for all positive integers $n$ and for each $x\in D_{[0,1]}$. In fact, by integrating by parts and using the fact that for each $x\in D_{[0,1]}$ we have $x(0)=x(1)=0$  and
  \begin{eqnarray*}
    (ABx)(t)&=&\int\limits_0^1 \frac{1}{\ln 2}\cdot\frac{t+1}{ (s+1)^2 }\left(-\ln 2 (s+1)\right)x'(s)ds=-\int\limits_{0}^{1} \frac{t+1}{ (s+1)^2 }x(s)ds \\
            &=& (-\ln 2) \,(Ax)(t),\\
            (A^2x)(t)&=&\int\limits_0^1 \frac{1}{\ln 2}\cdot\frac{t+1}{ (s+1)^2 } \left(\int\limits_0^1 \frac{1}{\ln 2}\cdot\frac{s+1}{ (\tau+1)^2 }x(\tau) d\tau\right)ds\\
                  &=& \int\limits_0^1 \frac{1}{\ln 2}\cdot \frac{t+1}{ (\tau+1)^2 }x(\tau)d\tau =(Ax)(t), \\
    (BAx)(t)&=&  -\ln 2\, (t+1) \int\limits_{0}^{1} \frac{\left(t+1\right)'_t}{\ln 2(s+1)^2}x(s)ds=-\int\limits_{0}^{1}\frac{t+1}{ (s+1)^2 }x(s)ds\\ &=&(-\ln 2)\, (Ax)(t).
  \end{eqnarray*}
  Therefore for each $x\in D_{[0,1]}$ we have $ABx=BA^nx$ for all positive integers $n$.
  }\end{ex}

  \begin{ex}\label{ExampleABBA2IntDiffOpABNotCommute} {\rm
   We seek for a particular solution from the following relations
   \begin{eqnarray*}
     b(t)&=&\frac{a(t)\lambda}{\delta Q_{[\alpha,\beta]}(a,c)^{n-1} a'(t)},\\
     c(s)&=&{\rm exp}\left(\di\int\limits_{[\xi_0,s]\cap\,\Omega_a\,\cap\, \Omega_c} \frac{- \delta \mu^{n-1} a'(\tau)^2+a(\tau)a''(\tau)-a'(\tau)^2}{a(\tau)a'(\tau)}d\tau\right),
   \end{eqnarray*}
   for some  $\lambda\in\mathbb{R}\setminus\{0\}$, $\xi_0\in \Omega_a\cap \Omega_c$, $\Omega_a$ and $\Omega_c$ are defined in \eqref{SetOmegaaOmegacProofCorAdiffBIntOp}.
   We regard $b(t)=\lambda a(t)$, for some non-zero real constant $\lambda$, and look for $c$ such that $Q_{[\alpha,\beta]}(a,c)=\int\limits_{\alpha}^{\beta} a(s)c(s)ds=2$, $\alpha=0$, $\beta=1$ and $\delta=1$. Therefore, we get the following:
  \begin{gather*}
    b(t)=\lambda a(t)=\lambda\left(\frac{t}{2}+\gamma_0\right),\, \quad c(t)=\frac{(\xi_0+2\gamma_0)^3}{(t+2\gamma_0)^2},\quad t\in [0,1],
  \end{gather*}
  where $\lambda\in\mathbb{R}\setminus\{0\}$  and some $\gamma_0\in\mathbb{R}$.
  Thus, by applying Corollary \ref{corCorrespondingthmIntDiffOpCaseOfMonomial}, we have that operators $A:\,C^\infty[0,1]\to C^\infty[0,1]$, $B:\,C^\infty[0,1]\to C^\infty[0,1]$ defined as follows
  \begin{gather*}
    (Ax)(t)=  \int\limits_{0}^{1} \left(\frac{t}{2}+\gamma_0\right)\frac{(\xi_0+2\gamma_0)^3}{ (s+2\gamma_0)^3 }x(s)ds, \quad (Bx)(t)=\lambda \left(\frac{t}{2}+\gamma_0\right)\frac{dx}{dt},
  \end{gather*}
 where $\gamma_0$, $\xi_0$ satisfy $(2\gamma_0+\xi_0)^3=8\gamma_0(1+2\gamma_0)$, $\lambda$ is a real constant, satisfy the relation $ABx=BA^2x$ for each $x\in D_{[0,1]}$, where $D_{[0,1]}$ is given by \eqref{SubsetOfCinfinityZeroIntheboundary01}. In fact, by integrating by parts and using the fact that for each $x\in D_{[0,1]}$ we have $x(0)=x(1)=0$  and
  \begin{eqnarray*}
    (ABx)(t)&=&\int\limits_{0}^{1} \left(\frac{t}{2}+\gamma_0 \right)\frac{(\xi_0+2\gamma_0)^3}{(s+2\gamma_0)^3}\lambda\left(\frac{s}{2}+\gamma_0 \right)x'(s)ds \\
         &=&\lambda \int\limits_{0}^{1} \left(\frac{t}{2}+\gamma_0\right)\frac{(\xi_0+2\gamma_0)^3}{ (s+2\gamma_0)^3 }x(s)ds=\lambda (Ax)(t),\\
    (A^2x)(t)&=&\int\limits_{0}^{1} \left(\frac{t}{2}+\gamma_0\right)\frac{(\xi_0+2\gamma_0)^3}{ (s+2\gamma_0)^3 }\left(\int\limits_{0}^{1} \left(\frac{s}{2}+\gamma_0\right)\frac{(\xi_0+2\gamma_0)^3}{ (\tau+2\gamma_0)^3 }x(\tau)d\tau\right)ds\\
           &=&2\int\limits_{0}^{1} \left(\frac{t}{2}+\gamma_0\right)\frac{(\xi_0+2\gamma_0)^3}{ (\tau+2\gamma_0)^3 }x(\tau)d\tau=2(Ax)(t), \\
    (BAx)(t)&=&\lambda \left(\frac{t}{2}+\gamma_0 \right)\int\limits_{0}^{1} \left(\frac{t}{2}+ \gamma_0\right)'_t\frac{(\xi_0+2\gamma_0)^3}{(s+2\gamma_0)^3}ds=\frac{\lambda}{2}(Ax)(t)\\
    (BA^2x)(t)&=&2(BAx)(t)=\lambda (Ax)(t).
  \end{eqnarray*}
  Thus, for each $x\in D_{[0,1]}$ we have $ABx=BA^2x$.
}\end{ex}

\begin{remark}{\rm
Let $A:C^\infty[0,1]\to C^\infty[0,1]$ and $B:C^\infty[0,1]\to C^\infty[0,1]$ defined as follows
  \begin{gather*}
    (Ax)(t)=  \int\limits_{0}^{1} \left(\frac{t}{2}+\gamma_0\right)\frac{(\xi_0+2\gamma_0)^3}{ (s+2\gamma_0)^3 }x(s)ds, \quad (Bx)(t)=\lambda \left(\frac{t}{2}+\gamma_0\right)\frac{dx}{dt},
  \end{gather*}
where $\lambda\in \mathbb{R}\setminus\{0\}$ and $\gamma_0 \in\mathbb{R}$. By using results from Example \ref{ExampleABBA2IntDiffOpABNotCommute} we have for all $x\in D_{[0,1]}$, where $D_{[0,1]}$ is given in \eqref{SetDCinfinityXofalphabetazero},
\begin{eqnarray*}
  (ABx)(t)&=&\lambda (Ax)(t),\quad (BAx)(t)= \frac{\lambda}{2}(Ax)(t), \\
  (A^2x)(t)&=&\int\limits_{0}^{1} \left(\frac{t}{2}+\gamma_0\right)\frac{(\xi_0+2\gamma_0)^3}{ (s+2\gamma_0)^3 }\left(\int\limits_{0}^{1} \left(\frac{s}{2}+\gamma_0\right)\frac{(\xi_0+2\gamma_0)^3}{ (\tau+2\gamma_0)^3 }x(\tau)d\tau\right)ds\\
           &=&\varphi(\gamma_0,\xi_0)\int\limits_{0}^{1} \left(\frac{t}{2}+\gamma_0\right)\frac{(\xi_0+2\gamma_0)^3}{ (\tau+2\gamma_0)^3 }x(\tau)d\tau=\varphi(\gamma_0,\xi_0)(Ax)(t),
\end{eqnarray*}
where
\begin{gather*}
  \varphi(\gamma_0,\xi_0)=\int\limits_{0}^{1} \frac{(\xi_0+2\gamma_0)^3}{(s+2\gamma_0)^3}\left(\frac{s}{2}+\gamma_0\right)ds=\frac{(\xi_0+2\gamma_0)^3}{4\gamma_0(2\gamma_0+1)}.
\end{gather*}
Therefore, for all $x\in D_{[0,1]}$ operators $A$ and $B$ satisfy the commutation relation
\begin{gather*}
  (ABx)(t)-(BA^2x)(t)=\frac{\lambda}{2}(2-\varphi(\gamma_0,\xi_0))(Ax)(t).
\end{gather*}
 Since  $\lambda\not=0$ and if $A\not=0$ then  for all $x\in D_{[0,1]}$ the equality $(ABx)(t)=(BA^2x)(t)$ holds true  if and only if $\varphi(\gamma_0,\xi_0)=2$,
that is, $(\xi_0+2\gamma_0)^3=8\gamma_0(1+2\gamma_0)$. Moreover, for some constant $\tilde{\delta}$, some positive integer $\tilde{n}$ and for each $x\in D_{[0,1]}$ we have
\begin{gather*}
  (ABx)(t)-\tilde{\delta}(BA^{\tilde{n}}x)(t)= \left(1-\frac{[\varphi(\gamma_0,\xi_0)]^{\tilde{n}-1}}{2}\cdot \tilde{\delta}\right)     \lambda(Ax)(t).
\end{gather*}
}\end{remark}

\begin{lemma}\label{lemDiffIntOpBAn}
Consider linear operators $A:C^{\infty}[\alpha,\beta]\to C^{\infty}[\alpha,\beta]$,  $B:C^{\infty}[\alpha,\beta]\to C^{\infty}[\alpha,\beta]$
defined as follows
\begin{gather*}
 (Ax)(t)=a(t)\frac{dx}{dt},\qquad (Bx)(t)= \int\limits_{\alpha}^{\beta} k(t,s)x(s)ds,
\end{gather*}
where $\alpha,\beta$ are real numbers, $k(\cdot, \cdot)\in C^{\infty}([\alpha,\beta]^2)$, $ a(\cdot)\in C^{\infty}[\alpha,\beta]$.
We put
\begin{gather*}
  k_0(t,s)=k(t,s), \quad   
  k_m(t,s)=\frac{\partial }{\partial s}\left[a(s)k_{m-1}(t,s) \right],\quad m=1,\ldots,n. 
\end{gather*}
Then, for each non negative integer $n$ we have
\begin{gather*}
  (BA^nx)(t)=\sum_{i=0}^{n-1} (-1)^i k_i(t,s)a(s)(A^{n-1-i}x)(s)\Big|_{s=\alpha}^{s=\beta}+(-1)^n\int\limits_{\alpha}^{\beta} k_n(t,s)x(s)ds,
\end{gather*}
where
\begin{gather*}
  g(t,s)\Big|_{s=\alpha}^{s=\beta}=g(t,\beta)-g(t,\alpha) \mbox{ and } A^0\ \mbox{ is the identity operator}.
\end{gather*}
\end{lemma}

\begin{proof}
We proceed by induction. When $n=0$ it is trivial. For $n=1$, we integrate by parts. So we get
\begin{eqnarray*}
  (BAx)(t)&=&\int\limits_{\alpha}^{\beta} k(t,s) a(s)x'(s)ds=k(t,s)a(s)x(s) \Big|_{s=\alpha}^{s=\beta} -\int\limits_{\alpha}^{\beta} \frac{\partial}{\partial s}[k(t,s) a(s)] x(s)ds=\\
  &=&k_0(t,s)a(s)x(s)\Big|_{s=\alpha}^{s=\beta}-\int\limits_{\alpha}^{\beta} k_1(t,s) x(s)ds.
  \end{eqnarray*}



Suppose that for some integer $n>1$
\begin{gather*}
  (BA^nx)(t)=\sum_{i=0}^{n-1} (-1)^i k_i(t,s)a(s)(A^{n-1-i}x)(s)\Big|_{s=\alpha}^{s=\beta}+(-1)^n\int\limits_{\alpha}^{\beta} k_n(t,s)x(s)ds.
\end{gather*}
Then, we have
\begin{eqnarray*}
  (BA^{n+1}x)(t)&=&(BA^n)(Ax)(t)=\sum_{i=0}^{n-1} k_i(t,s)a(s)(A^{n-1-i})(Ax)(s)\Big|_{s=\alpha}^{s=\beta}+\\
 & & +(-1)^n\int\limits_{\alpha}^{\beta} k_n(t,s)(Ax)(s)ds
  =\sum_{i=0}^{n-1} (-1)^ik_i(t,s)a(s)(A^{n-i}x)(s)\Big|_{s=\alpha}^{s=\beta}+ \\
  & & +(-1)^n\int\limits_{\alpha}^{\beta} k_n(t,s)a(s)x'(s)ds
  =\sum_{i=0}^{n-1}(-1)^i k_i(t,s)a(s)(A^{n-i}x)(s)\Big|_{s=\alpha}^{s=\beta}+\\
  & &+(-1)^n k_n(t,s)a(s)x(s)\Big|_{s=\alpha}^{s=\beta}
  -(-1)^n\int\limits_{\alpha}^{\beta} k_{n+1}(t,s)x(s)ds=\\
  &=&\sum_{i=0}^{n} (-1)^i k_i(t,s)a(s)(A^{n-i}x)(s)\Big|_{s=\alpha}^{s=\beta}+(-1)^{n+1}\int\limits_{\alpha}^{\beta} k_{n+1}(t,s)x(s)ds.
\end{eqnarray*}
This completes the proof. \qed
\end{proof}

\begin{lemma}\label{lemRepreAB=BAnDiffIntOP}
Let $A:C^{\infty}[\alpha,\beta]\to C^{\infty}[\alpha,\beta]$ and  $B:C^{\infty}[\alpha,\beta]\to C^{\infty}[\alpha,\beta]$ be linear operators
defined as follows
\begin{gather*}
 (Ax)(t)=a(t)\frac{dx}{dt},\qquad (Bx)(t)= \int\limits_{\alpha}^{\beta} k(t,s)x(s)ds,
\end{gather*}
where $\alpha,\beta\in \mathbb{R}$, $k(\cdot, \cdot)\in C^{\infty}([\alpha,\beta]^2)$, $ a(\cdot)\in C^{\infty}[\alpha,\beta]$.
We put
\begin{gather*}
  k_0(t,s)=k(t,s), \quad   
  k_m(t,s)=\frac{\partial}{\partial s}\left[a(s)k_{m-1}(t,s) \right],\quad m=1,\ldots,n. 
\end{gather*}
Then for some positive integer $n$ and for all $x\in C^\infty[\alpha,\beta]$ the equality $(ABx)(t)=(BA^nx)(t)$ holds true  if and only if
\begin{eqnarray*}
\int\limits_{\alpha}^{\beta} \frac{\partial k(t,s)}{\partial t} x(s)ds&=&  \sum_{i=0}^{n-1} (-1)^i
k_i(t,s)a(s)(A^{n-1-i}x)(s)\Big|_{s=\alpha}^{s=\beta}+\\
&+&(-1)^n\int\limits_{\alpha}^{\beta} k_n(t,s)x(s)ds.
\end{eqnarray*}
\end{lemma}

\begin{proof}
 Let $x\in C^\infty{}[\alpha,\beta]$. By Lemma \ref{lemDiffIntOpBAn},
  \begin{gather*}
    (BA^nx)(t)=\sum_{i=0}^{n-1} (-1)^i k_i(t,s)a(s)(A^{n-1-i}x)(s)\Big|_{s=\alpha}^{s=\beta}+(-1)^n\int\limits_{\alpha}^{\beta} k_n(t,s)x(s)ds.
   \end{gather*}
Since $(ABx)(t)=\frac{d}{dt}\left(\int\limits_{\alpha}^{\beta} k(t,s)x(s)ds \right)=\int\limits_{\alpha}^{\beta} \frac{\partial k(t,s)}{\partial t} x(s)ds,$ then $(ABx)(t)=(BA^nx)(t) $ for all $x\in C^\infty[\alpha,\beta]$ if and only if
  \begin{eqnarray*}
\int\limits_{\alpha}^{\beta} \frac{\partial k(t,s)}{\partial t} x(s)ds&=&  \sum_{i=0}^{n-1} (-1)^i
k_i(t,s)a(s)(A^{n-1-i}x)(s)\Big|_{s=\alpha}^{s=\beta}+ \\
&+& (-1)^n\int\limits_{\alpha}^{\beta} k_n(t,s)x(s)ds.
  \end{eqnarray*}
This completes the proof. \qed
\end{proof}

\begin{theorem}\label{ThmRepDiffIntOp}
Let $A:C^{\infty}[\alpha,\beta]\to C^{\infty}[\alpha,\beta]$ and  $B:C^{\infty}[\alpha,\beta]\to C^{\infty}[\alpha,\beta]$ be linear operators defined as
\begin{gather*}
 (Ax)(t)=a(t)\frac{dx}{dt},\qquad (Bx)(t)= \int\limits_{\alpha}^{\beta} k(t,s)x(s)ds,
\end{gather*}
where $\alpha,\beta\in\mathbb{R}$, $k(\cdot, \cdot)\in C^{\infty}([\alpha,\beta]^2)$, $ a\in C^{\infty}[\alpha,\beta]$
such that $a(\alpha)=a(\beta)=0$.
Consider a polynomial $F(t)=\sum\limits_{i=1}^{n}\delta_i t^i$, where $\delta_1,\ldots,\delta_n\in\mathbb{R}$.
Put
\begin{gather*}
  k_0(t,s)=k(t,s) \quad   
  k_m(t,s)=\frac{\partial}{\partial s}\left[a(s)k_{m-1}(t,s) \right],\quad m=1,\ldots,n. 
\end{gather*}
Then $AB=BF(A)$ if and only if

  \begin{gather}\label{CondMainThmNessSufCondAB=BFADiffIntOP}
  \left\{\begin{array}{c}\di
           a(t)\frac{\partial k(t,s)}{\partial t}=\sum_{m=1}^{n}(-1)^{m} \delta_{m} k_{m}(t,s)  \\ \\ \di
           a(\alpha )=a(\beta)=0.
         \end{array} \right.
\end{gather}
\end{theorem}

\begin{proof}
Lemma \ref{lemDiffIntOpBAn}, Lemma \ref{lemRepreAB=BAnDiffIntOP} and $a(\alpha)=a(\beta)=0$ imply that
 $AB=BF(A)$ if and only if
 \begin{gather}\label{ProofCondiThmReprAdiffOpB}
   \forall x\in C^\infty[\alpha,\beta],\quad
\int\limits_{\alpha}^{\beta} \frac{\partial k(t,s)}{\partial t} x(s)ds= \int\limits_{\alpha}^{\beta} \sum_{m=1}^{n} (-1)^m \delta_m k_m(t,s)x(s)ds.
 \end{gather}
By \cite[Corollary 4.23]{BrezisFASobolevSpaces} and Lemma \ref{LemmaCLFVanishCompleteSetVanishUniverse}, the condition
 \eqref{ProofCondiThmReprAdiffOpB} is equivalent to \eqref{CondMainThmNessSufCondAB=BFADiffIntOP}.
\qed
\end{proof}

\begin{ex}{\rm
Consider operators $A: C^\infty[0,1]\to C^\infty[0,1]$ and $B:C^\infty[0,1]\to C^\infty[0,1]$ defined as follows
\begin{gather*}
  (Ax)(t)=a(t)\frac{dx}{dt},\quad (Bx)(t)=\int\limits_{0}^{1}b(t)c(s)x(s)ds,
\end{gather*}
where $a,b,c\in C^\infty[0,1]$. We will look for functions $a,b,c$ such that $AB=BA^2$. According to Theorem \ref{ThmRepDiffIntOp}, we have $AB=BA^2$ if and only if
\begin{gather*}
  \left\{\begin{array}{l}\di
           a(t)\frac{\partial b(t)c(s)}{\partial t}= (-1)^2 \frac{\partial}{\partial s}\big(a(s)\frac{\partial}{\partial s}(a(s)b(t)c(s))\big)  \\ \\
        a(0)=a(1)=0
         \end{array} \right.
\end{gather*}
So we have
\begin{gather*}
\left\{\begin{array}{cc}  a(t)b'(t)c(s)=b(t)( 3a(s)a'(s)c'(s)+a'(s)^2c(s)+a''(s)a(s)c(s)+a(s)^2c''(s)) \\ \\
           a(0)=a(1)=0
 \end{array}\right.
\end{gather*}
Suppose that $a(t)\neq0$,  $b(t)\neq0$, $c(t)\neq0$ for all $t\in ]0,1[$. Then
 we have
\begin{equation*}
 \frac{ a(t)b'(t)}{b(t)}=\frac{ 3a(s)a'(s)c'(s)+a'(s)^2c(s)+a''(s)a(s)c(s)+a(s)^2c''(s)}{c(s)}=\lambda \end{equation*}
where $\lambda$ is a real constant. We get the following equations:
\begin{gather} \label{FunctionbFromExDiffIntOp}
  \frac{\lambda}{a(t)}=\frac{b'(t)}{ b(t)}, \\ \label{FunctionsacFromExDiffIntOp}
  \frac{ 3a(s)a'(s)c'(s)+a'(s)^2c(s)+a''(s)a(s)c(s)+a(s)^2c''(s)}{c(s)}=\lambda.
\end{gather}
From \eqref{FunctionbFromExDiffIntOp} we get $b(t)=\lambda_1\exp(\int\frac{\lambda}{a(t)}dt)$, $\lambda_1$ is a constant. By putting $e(t)=\frac{c'(t)}{c(t)}$ in \eqref{FunctionsacFromExDiffIntOp}, this reduces to
\begin{equation*}
  a(s)^2e'(s)+3a(s)a'(s)e(s)+a(s)^2e(s)^2+a'(s)^2+a''(s)a(s)-\lambda=0.
\end{equation*}
If $\lambda<0$, $a(t)=\sqrt{\lambda t(t-1)},$ $0<t<1,$ then $$e(s)=\frac{1}{(s(s-1))^{3/2}\lambda_2-4s^3+6s^2-2s},\ 0<s<1,$$
where $\lambda_2\in\mathbb{R}\setminus\{0\}$. Therefore,
$e(s)=\frac{c'(s)}{c(s)}\Longrightarrow c(s)=\lambda_3 \exp\big(\int e(s)ds\big),$
where $\lambda_3\in\mathbb{R}$.
}\end{ex}

\section*{Acknowledgments}
This work was supported by the Swedish International Development Cooperation
Agency (Sida) bilateral program with  Mozambique. Domingos thanks to Dr. Yury Nepomnyashchkh for useful comments  and is also grateful to the
Mathematics and Applied Mathematics research environment,
Division of Mathematics and Physics, School of Education, Culture and Communication,
M\"alardalens University for creating a research and educational environment.


\begin{thebibliography}{99}

\bibitem{AdamsG}  Adams, M. and Gullemin, V.: Measure Theory and Probability. Birkh\"auser (1996)

\bibitem{AkkLnearOperators} Akhiezer, N. I., Glazman, I. M.: Theory of Linear Operators in Hilbert Spaces. Volume I, Pitman Advanced Publ. (1981)


\bibitem{AdamsSobolevSpaces}  Adams, R. A.: Sobolev Spaces. Pure and Applied Mathematics series, Academic Press, ISBN 0-12-044150-0 (1975)

\bibitem{BratEvansJorg2000}
Bratteli, O., Evans, D. E., Jorgensen, P. E. T.: Compactly supported wavelets and representations of the {C}untz
relations. Appl. Comput. Harmon. Anal. 8(2), 166--196 (2000)

\bibitem{BratJorgIFSAMSmemo99}
Bratteli, O., Jorgensen, P. E. T.: Iterated function systems and permutation representations of the Cuntz algebra. Mem. Amer. Math. Soc. \textbf{139}(663), x+89 (1999)

\bibitem{BratJorgbook}
Bratteli, O., Jorgensen, P. E. T.: Wavelets through a looking glass. The world of the spectrum. Applied and Numerical Harmonic Analysis. Birkhauser Boston, Inc., Boston, MA, xxii+398 (2002)


\bibitem{BrezisFASobolevSpaces} Brezis, H.:  Functional Analysis, Sobolev Spaces and Partial Differential Equations. Springer, New York (2011)


\bibitem{CarlsenSilvExpoMath07}
Carlsen, T. M., Silvestrov, S.: $C^*$-crossed products and shift spaces, Expo. Math. \textbf{25}, no. 4, 275--307 (2007)

\bibitem{CarlsenSilvAAM09} Carlsen, T. M., Silvestrov, S.: On the Exel crossed product of topological covering maps.
Acta Appl. Math. \textbf{108}, no. 3, 573--583 (2009)

\bibitem{CarlsenSilvProcEAS10} Carlsen, T. M., Silvestrov, S.: On the $K$-theory of the $C^*$-algebra associated with a
one-sided shift space. Proc. Est. Acad. Sci. \textbf{59}, no. 4, 272--279 (2010)

\bibitem{ConwayFunctionalAnalysis} Conway, J. B.:
A course in functional analysis, 2nd ed. Graduate texts in mathematics {\bf 96}. 

\bibitem{JSvT12a}
De Jeu, M., Svensson, C.,  Tomiyama, J.: On the Banach $*$-algebra crossed product associated with a topological dynamical system. J. Funct. Anal. \textbf{262}(11), 4746--4765 (2012)

\bibitem{JSvT12b}
De Jeu, M.,  Tomiyama, J.: Maximal abelian subalgebras and projections in two Banach algebras associated with a topological dynamical system. Studia Math. \textbf{208}(1), 47--75 (2012)

\bibitem{DjinjaEtAll_LinMultIntOp} Djinja, D., Silvestrov, S., Tumwesigye, A. B. { Multiplication and linear integral operators in $L_p$ spaces representing polynomial covariance type commutation relations}. To appear in: A. Malyarenko, S. Silvestrov (eds) Non-commutative and Non-associative Algebra and Analysis Structures, SPAS 2019. Springer Proceedings in Mathematics and Statistics, vol. {426}, Springer, (2023)

\bibitem{DjinjaEtAll_IntOpOverMeasureSpaces}
Djinja, D., Silvestrov, S., Tumwesigye, A.B.: Representations of Polynomial Covariance Type Commutation Relations by Linear Integral Operators on $L_p$
 Over Measure Spaces. In: Malyarenko, A., Ni, Y., Ran\u{c}i\'{c}, M., Silvestrov, S. (eds) Stochastic Processes, Statistical Methods, and Engineering Mathematics . SPAS 2019. Springer Proceedings in Mathematics and Statistics, vol.  {408}., Ch. 4, Springer, pp. 59-95, (2022)

\bibitem{DjinjaEtAll_LinItOpInGenSepKern} Djinja, D., Silvestrov, S., Tumwesigye, A. B.: Representations of polynomial covariance type commutation relations by linear integral operators with general separable kernels in $L_p$ spaces. arXiv:2305.04144 [math.FA]  (2023)



\bibitem{DutkayJorg3}
Dutkay, D. E., Jorgensen, P. E. T.: Martingales, endomorphisms, and covariant systems of operators in Hilbert space. J. Operator Theory, \textbf{58}(2), 269--310 (2007)

\bibitem{DJS12JFASilv}
Dutkay, D. E., Jorgensen, P. E. T., Silvestrov, S.: Decomposition of wavelet representations and Martin boundaries. J. Funct. Anal. \textbf{262}(3), 1043--1061 (2012). (arXiv:1105.3442 [math.FA], 2011)

\bibitem{DLS09}
Dutkay, D. E., Larson, D. R., Silvestrov, S: Irreducible wavelet representations and ergodic automorphisms on solenoids. Oper. Matrices \textbf{5}(2), 201--219 (2011)
(arXiv:0910.0870 [math.FA], 2009)

\bibitem{DutSilvProcAMS}
Dutkay, D. E., Silvestrov, S.: Reducibility of the wavelet representation associated to the Cantor set. Proc. Amer. Math. Soc. \textbf{139}(10),  3657--3664 (2011). (arXiv:1008.4349 [math.FA], 2010).

\bibitem{DutSilvSV} Dutkay, D. E., Silvestrov, S.: Wavelet Representations and Their Commutant.
In: {\AA}str{\"o}m, K., Persson, L-E., Silvestrov, S. (eds.): Analysis for science, engineering and beyond. Springer proceedings in Mathematics Vol. 6, Springer, Berlin, Heidelberg, Ch. 9, pp 253-265 (2012)


\bibitem{FollandRA} Folland, G.: Real Analysis: Modern techniques and their applications. 2nd ed, John Wiley $\&$ Sons Inc. (1999)

\bibitem{HutsonPym} Hutson, V., Pym, J. S., Cloud, M. J.: Applications of Functional Analysis and Operator Theory. 2nd edition, Elsevier (2005)

\bibitem{JorgWavSignFracbook}
Jorgensen, P. E. T.: Analysis and probability: wavelets, signals, fractals. Graduate Texts in Mathematics, 234. Springer, New York, xlviii+276 (2006)

\bibitem{JorgOpRepTh88}
Jorgensen, P. E. T.: Operators and Representation Theory. Canonical Models for Algebras of Operators Arising in Quantum Mechanics. North-Holand Mathematical Studies 147 (Notas de Matem{\'a}tica 120), Elsevier Science Publishers, viii+337 (1988)

\bibitem{JorMoore84}
Jorgensen, P. E. T., Moore, R. T.: Operator Commutation Relations.
Commutation Relations for Operators, Semigroups, and Resolvents with Applications to Mathematical Physics and Representations of Lie Groups. Springer Netherlands, xviii+493 (1984)

\bibitem{Kantarovitch} Kantorovitch, L. V., Akilov, G. P.: Functional Analysis. 2nd ed, Pergramond Press Ltd, England (1982)


\bibitem{Kolmogorov} Kolmogorov, A. N., Fomim, S. V.: Elements of the theory of functions and Functional Analysis. 1st vol, Graylock press. (1957)

\bibitem{KolmogorovVol2} Kolmogorov, A. N. and Fomim, S. V.: Elements of the theory of functions and Functional Analysis. 2nd vol, Graylock press (1961)

\bibitem{KrasnolskZabreyko} Krasnosel'skii, M.A., Zabreyko P.P., Pustylnik E.I., Sobolevski P.E.:
Integral Operators on the space of summable functions. Springer Netherlands, Noordhoff Int. Publ. (1976)


\bibitem{MACbook1} Mackey, G. W.: Induced Representations of Groups and
Quantum Mechanics. W. A. Benjamin, New York, Editore Boringhieri, Torino (1968)

\bibitem{MACbook2} Mackey, G. W.: The Theory of Unitary Group
Representations. University of Chicago Press (1976)

\bibitem{MACbook3} Mackey, G. W.: Unitary Group Representations in Physics, Probability, and Number Theory. Addison-Wesley (1989)

\bibitem{Mansour16}
Mansour, T., Schork, M.: Commutation Relations, Normal Ordering, and Stirling Numbers,  CRC Press (2016)

\bibitem{JMusondaPhdth18} Musonda, J.:
Reordering in Noncommutative Algebras, Orthogonal Polynomials and Operators. PhD thesis, M\"alardalen University, (2018)

\bibitem{JMusonda19} Musonda, J., Richter, J., Silvestrov, S.:
Reordering in a multi-parametric family of algebras. Journal of Physics: Conference Series. \textbf{1194}, 012078 (2019)

\bibitem{Musonda20}
Musonda, J., Richter, J., Silvestrov, S.: Reordering in noncommutative algebras associated with iterated function systems. In: Silvestrov, S., Malyarenko, A., Ran\u{c}i\'{c}, M. (eds.), Algebraic structures and applications, Springer Proceedings in Mathematics and Statistics, vol. 317, Springer (2020)

\bibitem{Nazaikinskii96} Nazaikinskii,  V. E., Shatalov, V. E., Sternin,  B. Yu.:
Methods of Noncommutative Analysis. Theory and Applications.
De Gruyter Studies in Mathematics 22 Walter De Gruyter \& Co. Berlin (1996)

\bibitem{OstSambook}
Ostrovsky\u{\i}, V. L.,  Samo\u{\i}lenko, Yu. S.:
Introduction to the Theory of Representations of Finitely Presented $*$-Algebras. I. Representations by bounded operators.
Rev. Math. Phys. \textbf{11}. The Gordon and Breach Publ. Group (1999)

\bibitem{Pedbook79}  Pedersen, G. K.: $C^*$-algebras and their automorphism groups. Academic Press (1979)

\bibitem{PerssonSilvestrov031} Persson, T., Silvestrov, S. D.:
From dynamical systems to commutativity in non-commutative operator algebras.
In: A. Khrennikov (ed.), Dynamical systems from number theory to probability - 2, V{\"a}xj{\"o} University Press, Mathematical Modeling in Physics, Engineering and Cognitive Science, vol. 6, 109–143 (2003)

\bibitem{PerssonSilvestrov032} Persson, T., Silvestrov, S. D.:
Commuting elements in non-commutative algebras associated to dynamical systems.
In: A. Khrennikov (ed.), Dynamical systems from number theory to probability - 2, V{\"a}xj{\"o} University Press, Mathematical Modeling in Physics, Engineering and Cognitive Science, vol. 6, 145–172 (2003)

\bibitem{PersSilv:CommutRelDinSyst}  Persson, T., Silvestrov, S. D.:
Commuting operators for representations of commutation relations defined by dynamical systems.
Numerical Functional Analysis and Optimization \textbf{33}(7-9), 1146-1165 (2002)

\bibitem{RST16}
Richter, J.,  Silvestrov, S., Tumwesigye, A. B.: Commutants in crossed product algebras for piece-wise constant functions. In: Silvestrov S., Ran\v{c}i\'{c} M., (eds.), Engineering Mathematics II: Algebraic, Stochastic and Analysis Structures for Networks, Data Classification and Optimization. Springer Proceedings in Mathematics and Statistics, vol 179, Springer, 95-108 (2016)

\bibitem{RSST16}
Richter, J., Silvestrov S., Ssembatya V., Tumwesigye, A. B.: Crossed product algebras for piece-wise constant functions, Silvestrov S., Ran\v{c}i\'{c} M., (eds.), Engineering Mathematics II: Algebraic, Stochastic and Analysis Structures for Networks, Data Classification and Optimization. Springer Proceedings in Mathematics and Statistics, Vol. 179, Springer, 75-93 (2016)

\bibitem{RudinRCA} Rudin, W.: Real and Complex Analysis. 3rd ed, Mc Graw-Hill (1987)

\bibitem{RynneLFA} Rynne, B. P. and Youngson. M. A.: Linear Functional Analysis, 2nd ed, Springer (2008)

\bibitem{Samoilenkobook} Samoilenko, Yu. S.: Spectral theory of families of self-adjoint operators.
Kluwer Academic Publ. (1991) (Extended transl. from Russian edit. published by Naukova Dumka, Kiev, 1984)

\bibitem{SaV8894} Samoilenko, Yu. S., Vaysleb, E. Ye.:
On representation of relations $AU=UF(A)$ by unbounded self-adjoint and unitary operators. In: Boundary Problems for Differential Equations. Acad. of Sciences of Ukrain. SSR, Inst. Mat., Kiev, 30-52 (1988) (Russian). English transl.: Representations of the relations $AU=UF(A)$ by unbounded self-adjoint and unitary operators. Selecta Math. Sov. \textbf{13}(1), 35-54 (1994)

\bibitem{SilPhD95} Silvestrov, S. D.:
Representations of Commutation Relations. A Dynamical Systems Approach. Doctoral Thesis, Dep. of Math., Ume{\aa} University, \textbf{10}, (1995) (Hadronic Journal Supplement, \textbf{11}(1), 116 pp (1996))

\bibitem{STomdynsystype1} Silvestrov, S. D., Tomiyama, Y.: Topological dynamical systems of Type I, Expos. Math. \textbf{20}, 117-142 (2002)

\bibitem{SilWallin96} Silvestrov, S. D., Wallin, H.: Representations of algebras associated with a M{\"o}bius transformation.
J. Nonlin. Math. Physics \textbf{3}(1-2), 202-213 (1996)

\bibitem{SvSJ07a}  Svensson, C.,  Silvestrov, S., de Jeu, M.: Dynamical Systems and Commutants in Crossed Products.
Internat. J. Math. \textbf{18}, 455--471 (2007)

\bibitem{SvSJ07b}  Svensson, C.,  Silvestrov, S., de Jeu, M.:
Connections between dynamical systems and crossed products of Banach algebras by $\mathbb{Z}$.
In: Methods of spectral analysis in mathematical physics, 391--401, Oper. Theory Adv. Appl., vol. 186, Birkhäuser Verlag, Basel (2009) (Preprints in Mathematical Sciences, Centre for Mathematical Sciences, Lund University 2007:5, LUTFMA-5081-2007; Leiden Mathematical Institute report 2007-02; arXiv:math/0702118)

\bibitem{SvSJ07c}  Svensson, C.,  Silvestrov, S., de Jeu, M.:
Dynamical systems associated with crossed products. Acta Appl. Math.  \textbf{108}(3), 547--559 (2009)
(Preprints in Mathematical Sciences, Centre for Mathematical Sciences, Lund University 2007:22, LUTFMA-5088-2007; Leiden Mathematical Institute report 2007-30; arXiv:0707.1881 [math.OA]).

\bibitem{SvT09}  Svensson, C.,  Tomiyama, J.:
On the commutant of $C(X)$ in $C^*$-crossed products by $\mathbb{Z}$ and their representations. J. Funct. Anal.  \textbf{256}(7), 2367--2386 (2009)



\bibitem{Tomiyama87} Tomiyama, J.: Invitation to $C^*$-algebras and topological dynamics. World Scientific (1987)

\bibitem{Tomiama:SeoulLN1992}  Tomiyama, J.: The interplay between topological dynamics and theory of $C^*$-algebras.
Lecture Notes Series, \textbf{2}, Seoul National University Research Institute of Mathematics, Global Anal. Research Center, Seoul (1992)

\bibitem{Tomiama:SeoulLN2part2000}  Tomiyama, J.: The interplay between topological dynamics and theory of {C}$^*$-algebras. {II}., S\=urikaisekikenky\=usho K\=oky\=uroku (Kyoto Univ.) \textbf{1151}, 1--71 (2000)

\bibitem{AlexThesis2018} Tumwesigye, A. B.:  Dynamical Systems and Commutants in Non-Commutative Algebras. PhD thesis, M\"alardalen University, (2018)

\bibitem{VaislebSa90} Vaysleb, E. Ye., Samo\u{\i}lenko, Yu. S.: Representations of operator relations by unbounded operators
and multi-dimensional dynamical systems. Ukrain. Math. Zh. \textbf{42}(8), 1011-1019 (1990) (Russian). English transl.: Ukr. Math. J. \textbf{42} 899–906 (1990)

\end{thebibliography}
\end{document}